\documentclass[a4paper,11pt]{amsart}
\usepackage{color}
\usepackage{amsmath,amscd,amsthm,upref}
\usepackage[psamsfonts]{amssymb}
\usepackage{bm}
\usepackage[all]{xy}
\usepackage{enumerate}

\numberwithin{equation}{section}

\newtheorem{thm}{Theorem}[section]
\newtheorem{crl}[thm]{Corollary}
\newtheorem{lmm}[thm]{Lemma}
\newtheorem{prp}[thm]{Proposition}
\theoremstyle{definition}
\newtheorem{dfn}[thm]{Definition}
\newtheorem{ex}[thm]{Example}
\theoremstyle{remark}
\newtheorem*{rmk}{Remark}

\DeclareMathOperator{\Diff}{\bf Diff}
\DeclareMathOperator{\Top}{\bf Top}
\DeclareMathOperator{\NumG}{\bf NG}
\DeclareMathOperator{\STop}{\bf STop}

\DeclareMathOperator{\smap}{\bf smap}
\DeclareMathOperator{\Cinfty}{{\mathit C}^{\infty}}

\DeclareMathOperator{\Sd}{Sd}
\DeclareMathOperator{\Int}{Int}
\DeclareMathOperator{\Loops}{Loops}

\DeclareMathOperator*{\colim}{colim}

\newcommand{\R}{\mathbf{R}}
\newcommand{\I}{\mathcal{I}}
\newcommand{\J}{\mathcal{J}}
\newcommand{\K}{\mathcal{K}}
\newcommand{\semicolon}{\,;}
\newcommand{\bI}{\partial I}
\newcommand{\tI}{\tilde{I}}
\newcommand{\abs}[1]{\lvert{#1}\rvert}
\newcommand{\rel}[1]{\ \mbox{\rm rel}\ #1}

\title{A model structure on the category of diffeological spaces}
\author[T. Haraguchi]{Tadayuki Haraguchi} %
\address{Faculty of Education for Human Growth \\ Nara Gakuen
  University \\ Nara 636-8503 \\ Japan} %
\email{t-haraguchi@naragakuen-u.jp} %
\author[K. Shimakawa]{Kazuhisa Shimakawa} %
\address{Graduate School of Natural Science and Technology \\ Okayama
  University \\ Okayama 700-8530 \\ Japan} %
\email{kazu@math.okayama-u.ac.jp} \date{\today} %
\thanks{This work was supported by JSPS KAKENHI Grant Number
  JP18K03279.}
\begin{document}
\maketitle
\begin{abstract}
  We construct on the category of diffeological spaces a Quillen model
  structure having smooth weak homotopy equivalences as the class of
  weak equivalences.
  It is shown that our model structure on the category of
  diffeological spaces is Quillen equivalent to the standard Quillen
  model structure on the category of topological spaces, with weak
  homotopy equivalences as the class of weak equivalences.
\end{abstract}

\section{Introduction}
The theory of model category was developed by Quillen in \cite{QuiH}
and \cite{QuiR}.  By definition, a model category is just a category
with three specified classes of morphism, called fibrations,
cofibrations and weak equivalences, which satisfy several axioms that
are deliberately reminiscent of typical properties appearing in
homotopy theory of topological spaces.  It is shown in \cite[8.3]{Spa}
that the category $\bf Top$ of topological spaces has, so called, the
Quillen model structure, under which a map $f \colon X \to Y$ is
defined to be

\begin{enumerate}
\item a weak equivalence if $f$ is a weak homotopy equivalence
  \cite[p.404]{Spanier},
\item a fibration if $f$ is a Serre fibration {\cite{Serre}}, and
\item a cofibration if $f$ has the left lifting property with respect
  to trivial fibrations.
\end{enumerate}

In this article we prove that the category $\Diff$ of diffeological
spaces has a Quillen model category structure which is Quillen
equivalent to the standard Quillen model structure on $\Top$.
When we prove that $\Top$ is the model category under the requirements
stated above, we need knowledge of homotopy groups and cell complexes.
Likewise, we need to define and study smooth homotopy groups
and smooth version of cell complexes.

The paper is organized as follows.  We briefly review in Section 2 the
basic properties of diffeological spaces, following the treatment
given by Iglesias-Zemmour \cite{Zem}.

In Section 3 we introduce smooth homotopy groups $\pi_n(X,x_0)$ and
relative homotopy groups $\pi_n(X,A,x_0)$ as the sets of homotopy
classes of smooth maps $(I^n,\bI^n) \to (X,x_0)$ and
$(I^n,\bI^n,J^{n-1}) \to (X,A,x_0)$, respectively.  Here, $I^n$ is the
standard $n$-cube, $\bI^n$ is the boundary of $I^n$, and $J^{n-1}$ is
its subset $\bI^{n-1} \times I \cup I^{n-1} \times \{1\}$.
Our definition of homotopy groups is different from, but turns out to
be equivalent to, the one given by \cite{Zem}.
Due to the difficulty of finding a smooth retraction of $I^n$ onto
$J^{n-1}$, the treatment of smooth homotopy groups slightly harder
compared to the case of continuous homotopy groups.  Still, basic
properties of continuous homotopy groups mostly hold in the smooth
case; in particular, there exists a \emph{homotopy long exact
  sequence} associated with a pair of diffeological spaces.

In Section 4 we introduce the notion of a \emph{weak fibration} and
investigate its homotopical behavior with respect to cubical cell
complexes.
Briefly, weak fibrations are characterized by some sort of the right
lifting property with respect to the inclusions $L^{n-1} \to I^n$,
where $L^{n-1} = \bI^{n-1} \times I \cup I^{n-1} \times \{0\}$.
Although we cannot establish \emph{homotopy extension property} and
\emph{covering homotopy extension property} in its full generality, we
provide a limited version of them (Theorems~\ref{thm:WCHEP} and
\ref{thm:WHEP}), which are more than enough to construct commonly used
tools in homotopy theory such as \emph{homotopy long exact sequence
  for fibrations} and \emph{change of basepoint homomorphism}.

Based on the results obtained in preceding sections, we prove in
Section 5 that $\Diff$ has a model structure in which a weak
equivalence is a weak homotopy equivalence.
Unlike the case of topological spaces, our model structure is not
cofibrantly generated, so we have to straightforwardly verify the
axioms, following the argument given in \cite{Spa}.

In Section 6 we show that our model structure on $\Diff$ is left
Quillen equivalent to the standard Quillen model structure on $\Top$.
The proof boils down to showing that the unit $X \to DTX$, associated
with the adjunction between $\Diff$ and $\Top$, is a weak equivalence
for a certain type of smooth cell complexes $X$, which in turn is
derived from the classical Whitney approximation theorem on smooth
manifolds.

Finally, in Section 7 we compare the model structures of $\Top$ and
$\Diff$ with the model structure on the category $\NumG$ of
numerically generated spaces studied in \cite{Haraguchi}.  It turns
out that the model structures on $\Top$, $\NumG$, and $\Diff$ are
Quillen equivalent to each other.
We also observe that there exist diffeological spaces which have
smooth homotopy groups not isomorphic to their continuous homotopy
groups, and hence do not have the smooth homotopy type of a
topological space.  The Irrational torus is a typical example of such
spaces.

Throughout the paper, composition of maps is denoted in the form
$g \circ f$; but the abbreviated notation $g f$ is also used in the
complicated formulas.

The authors wish to thank Dan Christensen and Hiroshi Kihara for
helpful discussions while preparing the article.

\section{Diffeological spaces}
In this section we recall basic facts about diffeological spaces.  For
details see \cite{Zem}.

A diffeological space consists of a set $X$ together with a family $D$
of maps from open subsets of Euclidean spaces into $X$ satisfying the
following conditions:

(Covering) any constant parameterization $\R^n \to X$ belongs to $D$;

(Locality) a parameterization $P \colon U \to X$ belongs to $D$ if
every point $u$ of $U$ has a neighborhood $W$ such that $P|W \colon W
\to X$ belongs to $D$; and

(Smooth compatibility) if $P \colon U \to X$ belongs to $D$, then so
does the composite map $P\circ Q \colon V \to X$ for any smooth map $Q
\colon V \to U$ between open subsets of Euclidean spaces.

We call $D$ a diffeology of $X$, and each member of $D$ a plot of $X$.
Throughout the paper, $\R^n$ denotes the $n$-dimensional Euclidean
space equipped with the standard diffeology 
consisting of all smooth parameterizations of $\R^n$.

A map between diffeological spaces $f \colon X \to Y$ is called smooth
if for every plot $P \colon U \to X$ of $X$ the composite map
$f \circ P \colon U \to Y$ is a plot of $Y$.  In particular, if $D$
and $D'$ are diffeologies on a set $X$, then the identity map
$(X,D) \to (X,D')$ is smooth if and only if $D \subset D'$ holds.  In
that case, we say that $D$ is finer than $D'$, or $D'$ is coarser than
$D$.
A smooth map $f \colon X \to Y$ is called a diffeomorphism if $f$ is
bijective and its inverse $f^{-1}$ is also smooth.  If there is a
diffeomorphism from $X$ to $Y$ then we say that $X$ and $Y$ are
diffeomorphic and write $X \cong Y$.

Suppose $(X,D)$ is a diffeological space.  Then to any map $f$ from
a set $A$ to $X$ there exists a coarsest diffeology $f^*D$ on $A$
such that $f \colon A \to X$ is smooth.  The diffeology $f^*D$ is
called a pullback of $D$ along $f$.
A smooth injection $i \colon Z \to X$ is called an induction if $Z$
has the diffeology $i^*D$.  In particular, if $i$ is an inclusion of a
subset then $Z$ is called a subspace of $X$.

Dually, to any map $g$ from $X$ to a set $C$ there exists a finest
diffeology $g_*D$ on $C$ such that $g \colon X \to C$ is smooth.  The
diffeology $g_*D$ is called a pushforward of $D$ along $g$.  A smooth
surjection $p \colon X \to Z$ is called a subduction if $Z$ has the
diffeology $p_*D$.  In this case, $Z$ is called a quotient space of
$X$.
The following propositions are immediate from the definition.

\begin{prp}
  Let $f \colon X \to Y$ be a bijection between diffeological
  spaces. Then the following are equivalent to each other:
  \begin{enumerate}
  \item The map $f$ is an induction.
  \item The map $f$ is a subduction.
  \item The map $f$ is a diffeomorphism.
  \end{enumerate}
\end{prp}

\begin{prp}
  Let $X$, $Y$ and $Z$ be diffeological spaces and let
  $p \colon X \to Y$ be a subduction.  Then the following hold:
  \begin{enumerate}
  \item A map $f \colon Y \to Z$ is smooth if and only if so is $f
    \circ p$.
  \item A map $f \colon Y \to Z$ is a subduction if and only if so is
    $f \circ p$.
  \end{enumerate}
\end{prp}

The class of diffeological spaces together with smooth maps form a
category $\Diff$ which is complete, cocomplete, and is cartesian
closed.  (See, e.g.\ \cite[Theorem~2.1]{SYH}.)  In fact, its cartesian
closedness is a consequence of the exponential law described below.

For given diffeological spaces $X$ and $Y$, let $\Cinfty(X,Y)$ denote
the set of smooth maps from $X$ to $Y$.  Then $\Cinfty(X,Y)$ has a
diffeology $D_{X,Y}$ consisting of those parameterizations
$P \colon U \to \Cinfty(X,Y)$ such that for every plot
$Q \colon V \to X$ of $X$, the composition of $(P,Q)$ with the
evaluation map
\[
  U \times V \to \Cinfty(X,Y) \times X \to Y
\]
is a plot of $Y$.
In other words, $D_{X,Y}$ is the coarsest diffeology on $\Cinfty(X,Y)$
such that the evaluation map $\Cinfty(X,Y) \times X \to Y$ is smooth.


\begin{prp}[{\cite[1.59]{Zem}}]
  \label{prp:composition is smooth}
  Suppose $X$, $Y$ and $Z$ are diffeological spaces.  Then the
  composition
  \[
  \Cinfty(Y,Z) \times \Cinfty(X,Y) \to \Cinfty(X,Z), \quad (g,f)
  \mapsto g\circ f,
  \]
  is smooth.
\end{prp}

For given $f \in \Cinfty(X \times Y,Z)$, let us define
$\alpha(f) \in \Cinfty(X,\Cinfty(Y,Z))$ by the formula:
$\alpha(f)(x)(y) = f(x,y) \ (x \in X,\; y \in Y)$.  Then we have the
following exponential law.

\begin{prp}[{\cite[1.60]{Zem}}]
  \label{prp:exponential law}
  For any $X$, $Y$ and $Z$, the map
  \[
    \alpha \colon \Cinfty(X \times Y,Z) \to \Cinfty(X,\Cinfty(Y,Z)),
  \]
  which takes $f \in \Cinfty(X \times Y,Z)$ to $\alpha(f)$ is a
  diffeomorphism.
\end{prp}

\section{Homotopy sets}
We introduce homotopy groups for diffeological spaces in a slightly
different manner than the one given in \cite{Zem}.

Let $\R$ be the real line equipped with the standard diffeology, and
let $I$ be the unit interval $[0,1]$ equipped with the subspace
diffeology.  Suppose $f_0,\, f_1 \colon X \to Y$ are smooth maps
between diffeological spaces.  We say that $f_0$ and $f_1$ are
homotopic, written $f_0 \simeq f_1$, if there is a smooth map
$F \colon X \times I \to Y$ such that $F(x,0) = f_0(x)$ and
$F(x,1) = f_1(x)$ hold for every $x \in X$.  Such a smooth map $F$ is
called a homotopy between $f_0$ and $f_1$.
A map $f \colon X \to Y$ is called a homotopy equivalence if there is
a smooth map $g \colon Y \to X$ satisfying
\[
  g \circ f \simeq 1 \colon X \to X, \quad f \circ g \simeq 1 \colon Y
  \to Y.
\]
We say that $X$ and $Y$ are homotopy equivalent, written $X \simeq Y$,
if there exists a homotopy equivalence $f \colon X \to Y$.

We will show that the notion of homotopy introduced above is
equivalent to the one given in \cite{Zem}.
Following \cite{JM}, let $\gamma$ be the smooth function given by
$\gamma(t) = 0$ for $t \leq 0$, and $\gamma(t) = \exp(-1/t)$ for
$t > 0$, and put
\[
  \lambda(t) = \frac{\gamma(t)}{\gamma(t) + \gamma(1-t)}\, .
\]
Then $\lambda \colon \R \to I$ is a non-decreasing smooth function
satisfying $\lambda(t) = 0$ for $t \leq 0$, $\lambda(t) = 1$ for
$1\leq t$, and $\lambda(1-t) = 1-\lambda(t)$ for every $t$.

\begin{prp}
  Let $f_0,\, f_1 \colon X \to Y$ be smooth maps.  Then $f_0$ and
  $f_1$ are homotopic if and only if there exists a smooth map
  $G \colon X \times \R \to Y$ such that $G(x,0) = f_0(x)$ and
  $G(x,1) = f_1(x)$ hold for every $x \in X$.
\end{prp}

\begin{proof}
  Suppose there is a smooth map $G \colon X \times \R \to Y$ such that
  $G(x,0) = f_0(x)$ and $G(x,1) = f_1(x)$ hold for every $x \in X$.
  Then the restriction of $G$ to $X \times I$ gives a homotopy
  $f_0 \simeq f_1$.
  On the other hand, if there is a homotopy
  $F \colon X \times I \to Y$ between $f_0$ and $f_1$ then the
  composition $G = F\circ(1 \times \lambda)$ is a smooth map
  $X \times \R \to Y$ satisfying $G(x,0) = f_0(x)$ and
  $G(x,1) = f_1(x)$.
\end{proof}

Suppose $F$ is a homotopy between $f_0,\, f_1 \colon X \to Y$ and $G$
a homotopy between $f_1,\, f_2 \colon X \to Y$.  Let us define
$F * G \colon X \times I \to Y$ by the formula
\[
  F * G(x,t) =
  \begin{cases}
    F(x,\lambda(3t)), & 0 \leq t \leq 1/2,
    \\
    G(x,\lambda(3t-2)), & 1/2 \leq t \leq 1.
\end{cases}
\]
Then $F * G$ is smooth all over $X \times I$, hence gives a homotopy
between $f_0$ and $f_2$.  It follows that the relation ``$\simeq$'' is
an equivalence relation.  The resulting equivalence classes are called
homotopy classes.

In particular, if $P$ consists of a single point then smooth maps from
$P$ to $X$ are just the points of $X$ and their homotopies are smooth
paths $I \to X$.

\begin{dfn}
  Given a diffeological space $X$, we denote by $\pi_0{X}$ the set of
  path components of $X$, that is, equivalence classes of points of
  $X$, where $x$ and $y$ are equivalent if there is a smooth path
  $\alpha \colon I \to X$ such that $\alpha(0) = x$ and
  $\alpha(1) = y$ hold.
\end{dfn}

For given pairs of diffeological spaces $(X,X_1)$ and $(Y,Y_1)$, we
put
\[
  [X,X_1 \semicolon Y,Y_1] = \pi_0\Cinfty((X,X_1),(Y,Y_1)),
\]
where $\Cinfty((X,X_1),(Y,Y_1))$ is the subspace of $\Cinfty(X,Y)$
consisting of maps of pairs $(X,X_1) \to (Y,Y_1)$.
Similarly, we put
\[
  [X,X_1,X_2 \semicolon Y,Y_1,Y_2] =
  \pi_0\Cinfty((X,X_1,X_2),(Y,Y_1,Y_2)),
\]
where $\Cinfty((X,X_1,X_2),(Y,Y_1,Y_2))$ is the subspace consisting of
maps of triples.  Clearly, we have
$[X,X_1 \semicolon Y,Y_1] = [X,X_1,\emptyset \semicolon
Y,Y_1,\emptyset]$.

By the cartesian closedness of $\Diff$ we have the following.

\begin{prp}
  The elements of $[X,X_1,X_2 \semicolon Y,Y_1,Y_2]$ are in one-to-one
  correspondence with the homotopy classes of maps
  $(X,X_1,X_2) \to (Y,Y_1,Y_2)$.
\end{prp}

As a consequence of Propositions~\ref{prp:composition is smooth} and
\ref{prp:exponential law} together with the argument similar to that
of \cite[Proposition~6.1]{SYH}, we can show the following.

\begin{prp}
  \label{prp:mapping space is homotopy invariant}
  Suppose $f \colon (X,X_1,X_2) \to (Y,Y_1,Y_2)$ is a homotopy
  equivalence.  Then the precomposition and postcomposition by $f$
  induce homotopy equivalences
  \[
  \begin{split}
    f^* &\colon \Cinfty((Y,Y_1,Y_2),(Z,Z_1,Z_2)) \to
    \Cinfty((X,X_1,X_2),(Z,Z_1,Z_2))
    \\
    f_* &\colon \Cinfty((Z,Z_1,Z_2),(X,X_1,X_2)) \to
    \Cinfty((Z,Z_1,Z_2),(Y,Y_1,Y_2))
  \end{split}
  \]
  for every $(Z,Z_1,Z_2)$.
\end{prp}

\begin{proof}
  For any $\boldsymbol{X} = (X,X_1,X_2)$ and fixed
  $\boldsymbol{Z} = (Z,Z_1,Z_2)$, put
  \[
    F\boldsymbol{X} = \Cinfty(\boldsymbol{X},\boldsymbol{Z}) =
    \Cinfty((X,X_1,X_2),(Z,Z_1,Z_2)).
  \]
  We shall show that the contravariant functor $F$ from the category
  of triples of diffeological spaces to $\Diff$ preserves homotopies.
  This of course implies that
  $f^* \colon F\boldsymbol{Y} \to F\boldsymbol{X}$ is a homotopy
  equivalence if so is $f \colon \boldsymbol{X} \to \boldsymbol{Y}$.

  The contravariant functor $F$ is enriched in the sense that the map
  \[
    \Cinfty(\boldsymbol{X},\boldsymbol{Y}) \to
    \Cinfty(F\boldsymbol{Y},F\boldsymbol{X}),
  \]
  which takes $f \colon \boldsymbol{X} \to \boldsymbol{Y}$ to the
  induced map $f^* \colon F\boldsymbol{Y} \to F\boldsymbol{X}$, is
  smooth.  This follows from Proposition~\ref{prp:composition is
    smooth} because the map above is adjoint to the composition
  $\Cinfty(\boldsymbol{Y},\boldsymbol{Z}) \times
  \Cinfty(\boldsymbol{X},\boldsymbol{Y}) \to
  \Cinfty(\boldsymbol{X},\boldsymbol{Z})$.

  Suppose $h \colon \boldsymbol{X} \times I \to \boldsymbol{Y}$ is a
  homotopy between $f$ and $g$.  Then by
  Proposition~\ref{prp:exponential law} together with the enrichedness
  of $F$ the composite map
  \[
    I \to \Cinfty(\boldsymbol{X},\boldsymbol{Y}) \to
    \Cinfty(F\boldsymbol{Y},F\boldsymbol{X}),
  \]
  which takes $t \in I$ to
  $h_t^* \colon F\boldsymbol{Y} \to F\boldsymbol{X}$, is smooth.
  Thus, by passing to the adjoint again, we get a smooth map
  $F\boldsymbol{Y} \times I \to F\boldsymbol{X}$ giving a homotopy
  between $f^*$ and $g^*$.

  Quite similarly, we can prove that the covariant functor
  $\boldsymbol{X} \to \Cinfty(\boldsymbol{Z},\boldsymbol{X})$
  preserves homotopies.
\end{proof}

\begin{crl}
  \label{crl:homotopy set is homotopy invariant}
  The homotopy set $[X,X_1,X_2 \semicolon Y,Y_1,Y_2]$ is homotopy
  invariant with respect to both $(X,X_1,X_2)$ and $(Y,Y_1,Y_2)$.
\end{crl}

We are now ready to define the $n$-th homotopy set of a diffeological
space.  Denote by $\bI^n$ the boundary of the cube $I^n$, and let
\[
  J^{n-1} = \bI^{n-1} \times I \cup I^{n-1} \times \{1\}
\]
for $n \geq 1$.  We regard $J^{n-1}$ as a subspace of $\bI^n$.

\begin{dfn}
  Given a pointed diffeological space $(X,x_0)$, we put
  \[
    \pi_n(X,x_0) = [I^n,\bI^n \semicolon X,x_0],\quad n \geq 0.
  \]
  Similarly, given a pointed pair of diffeological spaces $(X,A,x_0)$,
  we put
  \[
    \pi_n(X,A,x_0) = [I^n,\bI^n,J^{n-1} \semicolon X,A,x_0],\quad n
    \geq 1.
  \]
\end{dfn}

For $n \geq 1$, $\pi_n(X,x_0)$ is isomorphic to $\pi_n(X,x_0,x_0)$,
and $\pi_0(X,x_0)$ is isomorphic to the set of path components
$\pi_0{X}$, regardless of the choice of basepoint $x_0$.  Note,
however, that we consider $\pi_0(X,x_0)$ as a pointed set with
basepoint $[x_0] \in \pi_0{X}$.

\begin{rmk}
  Our definition of $\pi_n(X,x_0)$ is equivalent to the one given in
  \cite{Zem}; in which the $n$-th homotopy set of $(X,x_0)$ is defined
  to be the set of path components of the iterated loop space
  $\Loops^n(X,x_0)$, where
  \[
    \Loops(Y,y) = \Cinfty((\R,0,1),(Y,y,y))
  \]
  for any pointed diffeological space $(Y,y)$.
  On the other hand, our $\pi_n(X,x_0)$ is the set of path components
  of another type of iterated loop space $\Omega^n(X,x_0)$, where
  \[
    \Omega(Y,y) = \Cinfty((I,0,1),(Y,y,y)).
  \]
  But the inclusion of $(I,0,1)$ into $(\R,0,1)$ is a homotopy
  equivalence because it has a homotopy inverse
  $\lambda \colon (\R,0,1) \to (I,0,1)$.
  Thus, by Proposition~\ref{prp:mapping space is homotopy invariant}
  we have
  \[
    \Loops(Y,y) \simeq \Omega(Y,y),
  \]
  implying the homotopy equivalence
  $\Loops^n(X,x_0) \simeq \Omega^n(X,x_0)$ for all $n \geq 0$.
  The situation is similar for the homotopy sets of pairs
  $\pi_n(X,A,x_0)$.
\end{rmk}

We now introduce a group structure on $\pi_n(X,A,x_0)$.  Suppose
$\phi$ and $\psi$ are smooth maps from $(I^n,\bI^n,J^{n-1})$ to
$(X,A,x_0)$.  If $n \geq 2$, or if $n \geq 1$ and $A = x_0$, then
there is a smooth map $\phi * \psi \colon I^n \to X$ which takes
$(t_1,t_2,\dots,t_n) \in I^n$ to
\[
  \begin{cases}
    \phi(\lambda(3t_1),t_2,\dots,t_n), & 0 \leq t_1 \leq 1/2
    \\
    \psi(\lambda(3t_1-2),t_2,\dots,t_n), & 1/2 \leq t_1 \leq 1.
  \end{cases}
\]
It is clear that $\phi * \psi$ defines a map of triples
$(I^n,\bI^n,J^{n-1}) \to (X,A,x_0)$, and there is a multiplication on
$\pi_n(X,A,x_0)$ given by the formula
\[
  [\phi] \cdot [\psi] = [\phi * \psi] \in \pi_n(X,A,x_0).
\]

\begin{prp}
  \label{prp:group structure}
  With respect to the multiplication
  $([\phi],[\psi]) \mapsto [\phi]\cdot[\psi]$, the homotopy set
  $\pi_n(X,A,x_0)$ is a group if $n \geq 2$ of if $n \geq 1$ and
  $A = x_0$, and is an abelian group if $n \geq 3$ of if $n \geq 2$
  and $A = x_0$.
  Moreover, for every
  smooth map $f \colon (X,A,x_0) \to (Y,B,y_0)$ the induced map
  \[
    f_* \colon \pi_n(X,A,x_0) \to \pi_n(Y,B,y_0)
  \]
  is a group homomorphism whenever its source and target are groups.
\end{prp}

In the case of topological spaces, the fact that $J^{n-1}$ is a
retract of $I^n$ is crucial for developing homotopy theory.  (Compare
e.g.\ homotopy exact sequence and homotopy extension property).
Unfortunately, it is not easy to construct a smooth retraction
$I^n \to J^{n-1}$.  Thus, we try to retrieve most of the ingredients
of homotopy theory without relying such a strict retraction.

\begin{dfn}
  \label{dfn:definition of tame maps}
  Let $f \colon K \to X$ be a smooth map from a cubical subcomplex $K$
  of $I^n$ (e.g.\ $I^n$, $\bI^n$, or $J^{n-1}$) to a diffeological
  space $X$.  Suppose $0 < \epsilon \leq 1/2$.  Then $f$ is called to
  be \emph{$\epsilon$-tame} if we have
  \[
    f(t_1,\cdots,t_{j-1},t_j,t_{j+1},\cdots,t_n) =
    f(t_1,\cdots,t_{j-1},\alpha,t_{j+1},\cdots,t_n)
  \]
  for every $(t_1,\cdots,t_n) \in K$ and $\alpha \in \{0,1\}$ such
  that $\abs{t_j - \alpha} \leq \epsilon$ holds.
  We use the abbreviation ``tame'' to mean $\epsilon$-tame for some
  $\epsilon > 0$.
\end{dfn}

Note that $\epsilon$-tameness implies $\sigma$-tameness for any
$\sigma < \epsilon$. Note also that $f \colon K \to X$ is $1/2$-tame
if and only if it is locally constant.  In particular, every map from
a $0$-dimensional complex $K$ is $1/2$-tame.

For $0 < \epsilon \leq 1/2$, we denote
$I^n(\epsilon) = [\epsilon,1-\epsilon]^n$ and call it the
\emph{$\epsilon$-chamber} of $I^n$.
More generally, if $K$ is a cubical subcomplex of $I^n$ then its
$\epsilon$-chamber $K(\epsilon)$ is defined to be the union of
$\epsilon$-chambers of its maximal faces.
Thus we have $\bI^n(\epsilon) = \bigcup F(\epsilon)$, where $F$ runs
through the $(n-1)$-dimensional faces of $I^n$, and
$J^{n-1}(\epsilon) = \bI^n(\epsilon) \cap J^{n-1}$.
It is evident that the following holds.

\begin{lmm}
  \label{lmm:uniqueness criterion for tame maps}
  Let $f$ and $g$ be smooth maps from a cubical subcomplex $K$ of
  $I^n$ to a diffeological space $X$.  Suppose both $f$ and $g$ are
  $\epsilon$-tame.  Then $f$ and $g$ coincide on $K$ if and only if
  they coincide on the $\epsilon$-chamber $K(\epsilon)$.
\end{lmm}

We show that any tame map defined on $J^{n-1}$ is extendable over
$I^n$.  To see this we need several lemmas.

\begin{lmm}
  \label{lmm:modified smash function}
  Suppose $0 \leq \sigma < \tau \leq 1/2$.  Then there exists a
  non-decreasing smooth function $T_{\sigma,\tau} \colon I \to I$
  satisfying the following conditions:
  \begin{enumerate}
  \item $T_{\sigma,\tau}(t) = 0$ for $t \leq \sigma$,
  \item $T_{\sigma,\tau}(t) = t$ for $\tau \leq t \leq 1 - \tau$,
  \item $T_{\sigma,\tau}(t) = 1$ for $1 - \sigma \leq t$, and
  \item $T_{\sigma,\tau}(1 - t) = 1 - T_{\sigma,\tau}(t)$
    for all $t$.
  \end{enumerate}
\end{lmm}

\begin{proof}
  For every $t \in \R$, put
  \[
    F(t) = \int_0^t \lambda\left(\frac{\tau x - \sigma}{\tau -
        \sigma}\right)dx +
    \frac{\tau+\sigma}{2\tau}\lambda\left(\frac{\tau t - \sigma}{\tau
        - \sigma}\right).
  \]
  Then $F \colon \R \to \R$ is a non-decreasing smooth function such
  that $F(t)$ has value $0$ for $t \leq \sigma/\tau$ and has value $t$
  for $t \geq 1$.  Now, let us define a function
  $T_{\sigma,\tau} \colon I \to I$ by putting
  $T_{\sigma,\tau}(t) = F(t/\tau)$ for $0 \leq t \leq 1/2$, and
  $T_{\sigma,\tau}(t) = 1 - F((1-t)/\tau)$ for $1/2 \leq t \leq 1$.
  As we have $T_{\sigma,\tau}(t) = t$ for $\tau \leq t \leq 1-\tau$,
  the function $T_{\sigma,\tau}$ is smooth all over $\R$ and satisfies
  the desired conditions.
\end{proof}

\begin{lmm}
  \label{lmm:taming of maps}
  Let $K$ be a cubical subcomplex of $I^n$.  Then for any smooth map
  $f \colon K \to X$ and $0 < \sigma < \epsilon \leq 1/2$, there
  exists a homotopy $f \simeq g$ relative to $K(\epsilon)$ such that
  $g$ is $\sigma$-tame.
  If $f$ is $\epsilon$-tame on a subcomplex $L$ of $K$ then the
  homotopy can be taken to be relative to $L \cup K(\epsilon)$.
\end{lmm}

\begin{proof}
  Let $g = f \circ (T_{\sigma,\epsilon})^n|K \colon K \to X$.  Then
  $g$ is $\sigma$-tame and there is a homotopy $f \simeq g$ relative
  to $K(\epsilon)$ given by the map $K \times I \to X$ which takes
  $(v,t) \in K \times I$ to $(1-t)f(v) + tg(v)$.
  If $f$ is $\epsilon$-tame on $L$ then the homotopy $f \simeq g$ is
  relative to $L \cup K(\epsilon)$ because $g$ coincides with $f$ on
  $L$.
\end{proof}

It follows, in particular, that any element of $\pi_n(X,A,x_0)$ is
represented by a tame map $(I^n,\bI^n,J^{n-1}) \to (X,A,x_0)$.

A map $I^n \to J^{n-1}$ is called an \emph{$\epsilon$-approximate
  retraction} if it restricts to the identity on the
$\epsilon$-chamber $J^{n-1}(\epsilon)$.

\begin{lmm}
  \label{lmm:approximate retraction exists}
  For any real number $\epsilon$ with $0 < \epsilon < 1/2$, there
  exists an $\epsilon$-approximate retraction
  $R_{\epsilon} \colon I^n \to J^{n-1}$.
\end{lmm}

\begin{proof}
  Let $\sigma < \epsilon' < \epsilon$, and let
  $m(u) = (1-u)\epsilon' + u\sigma$, so that we have
  $\epsilon' \geq m(u) \geq \sigma$ for $0 \leq u \leq 1$.
  For $t = (t_1,\cdots,t_{n-1}) \in I^{n-1}$ and $u \in I$, put
  \[
    v(t,u) = T_{\sigma,\epsilon}(u) + T_{\sigma,\epsilon}(1-u)\,
    {\prod}_{1 \leq k \leq n-1} \lambda\left(\frac{t_k}{m(u)}\right)
    \lambda\left(\frac{1-t_k}{m(u)}\right).
  \]
  Then the function $v \colon I^n \to I$ satisfies $v(t,u) = u$ for
  $(t,u) \in \bI^{n-1} \times I(\epsilon)$, and $v(t,u) = 1$ if
  $u \geq 1 - \sigma$ or if $t \in I^{n-1}(m(u))$ holds.
  It is easy to see that the smooth map
  $R_{\epsilon} \colon I^n \to J^{n-1}$ given by the formula
  \[
    R_{\epsilon}(t,u) =
    (T_{m(u),\epsilon}(t_1),\cdots,T_{m(u),\epsilon}(t_{n-1}),v(t,u)).
  \]
  restricts to the identity on $J^{n-1}(\epsilon)$.
\end{proof}

By combining this with Lemmas~\ref{lmm:uniqueness criterion for tame
  maps} we see that any tame map $J^{n-1} \to X$ can be extended to a
smooth map $I^n \to X$.  In fact, the following holds.

\begin{prp}
  \label{prp:tame maps are extendable}
  Any $\epsilon$-tame map $f \colon J^{n-1} \to X$ can be extended to
  a $\sigma$-tame map $g \colon I^n \to X$ for any
  $\sigma < \epsilon$.
  Moreover, if $f$ is $\epsilon'$-tame on $\bI^{n-1} \times \{0\}$
  then $g$ can be taken to be $\sigma'$-tame on $I^{n-1} \times \{0\}$
  for any $\sigma' < \epsilon'$.  
\end{prp}

\begin{proof}
  We may assume $\epsilon < \epsilon'$ and $\sigma < \sigma'$ hold.
  For $0 \leq u \leq 1$, let
  \[
    a(u) = \sigma + (\sigma' - \sigma)\lambda(1 - u/\sigma), \ \ %
    b(u) = \epsilon + (\epsilon' - \epsilon)\lambda(1 - u/\sigma).
  \]
  Then $a$ and $b$ are non-increasing functions satisfying
  $(a(0),b(0)) = (\sigma',\epsilon')$ and
  $(a(u),b(u)) = (\sigma,\epsilon)$ for $u \geq \sigma$.
  Now, let us define $g \colon I^n \to X$ by
  \[
    g(t_1,\,\cdots\,,t_{n-1},u) =
    f \circ R_{\epsilon}(T_{a(u),b(u)}(t_1),\,\cdots\,,
    T_{a(u),b(u)}(t_{n-1}), T_{\sigma,\epsilon}(u)).
  \]
  Then $g$ is a $\sigma$-tame extension of $f$ which also is
  $\sigma'$-tame on $I^{n-1} \times \{0\}$.
\end{proof}

For any pointed pair of diffeological spaces $(X,A,x_0)$, let
\[
  i_* \colon \pi_n(A,x_0) \to \pi_n(X,x_0), \quad %
  j_* \colon \pi_n(X,x_0) \to \pi_n(X,A,x_0)
\]
be the maps induced respectively by the inclusions
$(A,x_0) \to (X,x_0)$ and $(X,x_0,x_0) \to (X,A,x_0)$, and let
\[
  \varDelta \colon \pi_n(X,A,x_0) \to \pi_{n-1}(A,x_0) %
  \quad (n \geq 1)
\]
be the map which takes the class of
$\phi \colon (I^n,\bI^n,J^{n-1}) \to (X,A,x_0)$ to the class of its
restriction $\phi|I^{n-1} \colon (I^{n-1},\bI^{n-1}) \to (A,x_0)$.
Here, we identify $I^{n-1}$ with $I^{n-1} \times \{0\} \subset I^n$.
Clearly, $\varDelta$ is a group homomorphism for $n \geq 2$.

Since any element of the homotopy group has a tame representative, we
can obtain the homotopy exact sequence by arguing as in the case of
topological spaces.

\begin{prp}
  \label{prp:homotopy exact sequence}
  Given a pointed pair of diffeological spaces $(X,A,x_0)$, there is
  an exact sequence of pointed sets
  \begin{multline*}
    \cdots \xrightarrow{} \pi_{n+1}(X,A,x_0) \xrightarrow{\varDelta}
    \pi_n(A,x_0) \xrightarrow{i_*} \pi_n(X,x_0) \xrightarrow{j_*}
    \pi_n(X,A,x_0) \xrightarrow{} \cdots
    \\
    \cdots \xrightarrow{} \pi_1(X,A,x_0) \xrightarrow{\varDelta}
    \pi_0(A,x_0) \xrightarrow{i_*} \pi_0(X,x_0).
  \end{multline*}
\end{prp}

\section{Cubical complexes and fibrations}
We introduce a diffeological version of the notion of Serre fibration,
and describe its homotopical behavior with respect to cubical cell
complexes.

For $n \geq 1$, let
$L^{n-1} = \bI^{n-1} \times I \cup I^{n-1} \times \{0\} \subset I^n$.

\begin{dfn}
  \label{dfn:weak_fibration}
  A smooth map $p \colon E \to B$ is called a \emph{weak fibration} if
  for any pair of tame maps $f \colon L^{n-1} \to E$ and
  $g \colon I^n \to B$ satisfying $p \circ f = g|L^{n-1}$, there
  exists a smooth map $G \colon I^n \to E$ satisfying $G|L^{n-1} = f$
  and $p \circ G = g$.
\end{dfn}

Here the lift $G$ can be taken to be tame on $I^n$.  More precisely,
we have 

\begin{prp}
  \label{prp:tameness of lifts}
  In the definition of weak fibration above, suppose both $f$ and $g$
  are $\epsilon$-tame.  Then the lift $G \colon I^n \to E$ can be
  taken to be $\sigma$-tame for any $\sigma$ satisfying
  $0 < \sigma < \epsilon$.
\end{prp}

\begin{proof}
  Suppose $G' \colon I^n \to E$ satisfies $G'|L^{n-1} = f$ and
  $p \circ G' = g$, and let $\rho_n = T_{\sigma,\epsilon}^n$ for
  $0 < \sigma < \epsilon$.  As $f$ and $g$ are $\epsilon$-tame, we
  have $f \circ \rho_n|L^{n-1} = f$ and $g \circ \rho_n = g$.
  But then the composition $G = G' \circ \rho_n \colon I^n \to E$ is
  $\sigma$-tame and satisfies $G|L^{n-1} = f$ and $p \circ G = g$.
\end{proof}

\begin{ex}
  \label{ex:weak_fibrations}
  (1) It follows by Proposition~\ref{prp:tame maps are extendable}
  that for any diffeological space $X$ the constant map $X \to *$ is a
  weak fibration.

  (2) If $p \colon E \to B$ is a diffeological fiber bundle with fiber
  $F$ then its pullback by a smooth map from $I^n$ to $B$ is trivial
  (cf.\ \cite[8.19, Lemma 2]{Zem}).  But (1) implies that a trivial
  fiber bundle is a weak fibration, hence so is $p$.

  (3) Given a diffeological space $X$ with basepoint $x_0$, let
  $P(X,x_0)$ denote the subset of $\Cinfty((I,\{1\}),(X,x_0))$
  consisting of those tame paths $\ell \colon I \to X$ satisfying
  $\ell(1) = x_0$.  Then we can show by using
  Lemma~\ref{lmm:approximate retraction exists} that the map
  $p \colon P(X,x_0) \to X$, which takes a path $\ell$ to its initial
  point $\ell(0)$, is a weak fibration.
  Note that the inclusion of $\widehat\Omega(X,x_0) = p^{-1}(x_0)$
  into $\Omega(X,x_0)$ is a homotopy equivalence by
  Lemma~\ref{lmm:taming of maps}.
\end{ex}

In the sequel we use the term \emph{cubical complex} to mean a
disjoint union of spaces, each diffeomorphic to a cubical subcomplex
of $I^n$ for various $n \geq 0$.
Thus a cubical complex $X$ is a union
$\bigcup_{\lambda \in \Lambda} X_{\lambda}$ such that each
$X_{\lambda}$ is diffeomorphic to $I^{n(\lambda)}$ for some
$n(\lambda) \geq 0$, and that the following conditions hold: (i) each
face of a cube $X_{\lambda}$ is a cube in $X$, (ii) the intersection
$X_{\lambda} \cap X_{\mu}$ of any two cubes is a face of each,
and (iii) each connected component of $X$ is a union of finite cubes.
%

\begin{dfn}
  \label{dfn:tame maps and tame homotopies}
  A smooth map $f$ from a cubical complex $X$ to a diffeological space
  $Y$ is called to be \emph{tame} if its restriction
  $f|X_{\lambda} \colon X_{\lambda} \to Y$ to each cube of $X$ is
  tame.
  Similarly, a homotopy $h \colon X \times I \to Y$ between tame maps
  is called to be \emph{tame} if
  $f|X_{\lambda} \times I \colon X_{\lambda} \times I \to Y$ is tame
  for every $\lambda$.
\end{dfn}

The theorem below shows that a weak fibration enjoys a slightly weaker
form of \emph{covering homotopy extension property} holds for cubical
complexes.

\begin{thm}
  \label{thm:WCHEP}
  Let $p \colon E \to B$ be a weak fibration, and $(X,A)$ be a pair of
  a cubical complex and its subcomplex.  Suppose there are a tame map
  $f \colon X \to E$ and tame homotopies $h \colon A \times I \to E$,
  $k \colon X \times I \to B$ satisfying $h_0 = f|A$,
  $k_0 = p\circ f$, and $p\circ h = k|A \times I$.
  Then there exists a tame homotopy $H \colon X \times I \to E$
  satisfying $H_0 = f$, $H|A \times I = h$, and $p\circ H = k$.
\end{thm}

Since the constant map $Y \to *$ is a weak fibration, we can apply the
theorem above to deduce a weaker form of \emph{homotopy extension
  property.}

\begin{thm}
  \label{thm:WHEP}
  Let $(X,A)$ be a pair of cubical complexes, and $f \colon X \to Y$
  be a tame map.
  Suppose there is a tame homotopy $h \colon A \times I \to Y$
  satisfying $h_0 = f|A$.  Then there exists a tame homotopy
  $H \colon X \times I \to Y$ satisfying $H_0 = f$ and
  $H|A \times I = h$.
\end{thm}

\begin{proof}[Proof of Theorem~\ref{thm:WCHEP}]
  For $n \geq -1$, let $X^n$ denote the union of all cubes
  $X_{\lambda}$ such that either $X_{\lambda} \subset A$ or
  $\dim X_{\lambda} \leq n$ holds.  Thus we have
  $A = X^{-1} \subset X^0 \subset X^1 \subset \cdots \subset X$.  We
  shall inductively construct a tame homotopy
  $H^n \colon X^n \times I \to E$ of $h|X^n \times I$ such that
  $H^n|X^{n-1} \times I = H^{n-1}$ holds.  The desired homotopy is
  obtained by putting $H = \colim H^n\colon X \times I \to E$.

  Suppose there is a tame homotopy
  $H^{n-1} \colon X^{n-1} \times I \to E$ satisfying
  $H^{n-1}_0 = f|X^{n-1}$, $H^{n-1}|A \times I = h$,
  $p \circ H^{n-1} = k|X^{n-1} \times I$, and choose a diffeomorphism
  $\Phi_{\lambda} \colon I^n \cong X_{\lambda}$ for each
  $n$-dimensional cube $X_{\lambda} \subset X^n$.
  Then we have a commutative diagram
  \[
    \xymatrix@C=46pt@R=32pt{%
      \bI^n \times I \cup I^n \times \{0\}
      \ar[r]
      \ar[d] & X^{n-1} \times I \cup X_{\lambda} \times \{0\}
      \ar[r]^-{H^{n-1} \cup f} \ar[d] & E \ar[d]^-{p}
      \\
      I^n \times I \ar[r]^-{\Phi_{\lambda} \times 1} \ar@{.>}[urr] &
      (X^{n-1} \cup X_{\lambda}) \times I \ar[r]^-{k} & B. }%
  \]
  Since the horizontal compositions of upper and lower arrows are
  tame, and since $p$ is a weak fibration, there exists a tame lift
  $K_{\lambda} \colon I^n \times I \to E$ making the diagram
  commutative.
  Hence there is a map
  $H_{\lambda} \colon (X^{n-1} \cup X_{\lambda}) \times I \to E$ such
  that the following diagram is commutative:
  \[
    \xymatrix@C=50pt{%
      X^{n-1} \times I \coprod I^n \times I \ar[r]^-{H^{n-1} \bigcup
        K_{\lambda}} \ar[d]_-{i \times 1 \bigcup \Phi_{\lambda} \times
        1} & E \ar[d]^-{p}
      \\
      (X^{n-1} \cup X_{\lambda}) \times I \ar[ru]_-{H_{\lambda}}
      \ar[r]^-{k} & B }%
  \]
  The map $H_{\lambda}$ is smooth, in fact tame, because so are
  $H^{n-1}$ and $K_{\lambda}$.
  Thus we can take the union of all such maps $H_{\lambda}$ to obtain
  a tame homotopy $H^n \colon X^n \times I \to E$ extending $H^{n-1}$
  and satisfying the desired conditions.
\end{proof}

By virtue of Theorems~\ref{thm:WCHEP} and \ref{thm:WHEP}, we can
construct basic tools in homotopy theory, such as homotopy long exact
sequence of a fibration and change of basepoint homomorphism, within
$\Diff$.

\begin{prp}
  Let $p \colon E \to B$ be a weak fibration and let $F = p^{-1}(b)$
  be the fiber at $b \in B$.  Let $i \colon F \to B$ be the inclusion.
  Then for any basepoint $e$ of $F$, there is an exact sequence
  \begin{multline*}
    \cdots \xrightarrow{} \pi_{n+1}(B,b) \xrightarrow{\varDelta'}
    \pi_n(F,e) \xrightarrow{i_*} \pi_n(E,e) \xrightarrow{p_*}
    \pi_n(B,b) \xrightarrow{} \cdots
    \\
    \cdots \xrightarrow{} \pi_1(B,b) \xrightarrow{\varDelta'}
    \pi_0(F,e) \xrightarrow{i_*} \pi_0(E,e) \xrightarrow{p_*}
    \pi_0(B,b).
  \end{multline*}
\end{prp}

\begin{proof}
  The fact that every element of $\pi_n(B,b)$ has a tame
  representative enables us to apply Theorem~\ref{thm:WCHEP} to show
  that
  \[
    p_* \colon \pi_n(E,F,e) \to \pi_n(B,b,b) = \pi_n(B,b)
  \]
  is bijective for $n \geq 1$.
  The desired exact sequence is obtained from the homotopy exact
  sequence for the pair $(E,F)$ together with the evident exact
  sequence
  $\pi_0(F,e_0) \xrightarrow{i_*} \pi_0(E,e_0) \xrightarrow{p_*}
  \pi_0(B,b_0)$.
\end{proof}

\begin{ex}
  \label{ex:irrational torus}
  Let $\theta$ be an irrational number.  The irrational torus
  $\mathbb{T}_{\theta}$ of slope $\theta$ is defined to be the
  quotient of the 2-torus $\mathbb{T}^2$ by the $1$-parameter subgroup
  $\R_{\theta} = \{[t,\theta t] \mid t \in \R\}$.
  By \cite[8.15]{Zem}, the natural map
  $p \colon \mathbb{T}^2 \to \mathbb{T}_{\theta}$ is a diffeological
  fiber bundle with contractible fiber $\R_{\theta}$.  By using the
  homotopy long exact sequence, we see that the map $p$ induces an
  isomorphism
  $\pi_n(\mathbb{T}^2,x) \cong \pi_n(\mathbb{T}_{\theta},p(x))$ for
  every $n \geq 0$ and $x \in \mathbb{T}^2$.  Thus $p$ is a weak
  homotopy equivalence, and hence $\mathbb{T}_{\theta}$ has the smooth
  fundamental group $\mathbb{Z}^2$.
\end{ex}

We now show that $\pi_n(X,x_0)$ can be identified with the set of
homotopy classes of smooth maps $(\bI^{n+1},e) \to (X,x_0)$, where
$e = (1,\cdots,1) \in \bI^{n+1}$.  Consider the commutative diagram
\[
  \xymatrix{%
    [I^n,\bI^n \semicolon X,x_0] & [\bI^{n+1},J^n \semicolon X,x_0]
    \ar[l]_-{i^*} \ar[r]^-{j^*} & [\bI^{n+1},e \semicolon X,x_0]
    \\
    [I^n/\bI^n,* \semicolon X,x_0] \ar[u]_-{\cong} & [\bI^{n+1}/J^n,*
    \semicolon X,x_0] \ar[l]_-{\cong} \ar[u]_-{\cong}
  }%
\]
induced by the evident inclusions and projections.  By the
commutativity of the left hand square, we see that $i^*$ is an
isomorphism.
To see that $j_*$ is an isomorphism, let us take a tame map
$f \colon (\bI^{n+1},e) \to (X,x_0)$.  Since $I^{n-1} \times \{1\}$ is
a deformation retract of $J^n$ and is contractible to $e$, there
exists a tame contracting homotopy $r \colon J^n \times I \to J^n$ of
$J^n$ onto $e$.
By applying Theorem~\ref{thm:WHEP} to the map $f$ and the homotopy
$f|J^n \circ r \colon J^n \times I \to X$, we obtain a homotopy
$f \simeq g$ relative to $e$ such that $g(J^n) = x_0$ holds.
Hence we have $[f] = j^*([g]) \in [\bI^{n+1},e \semicolon X,x_0]$,
implying that $j^*$ is surjective.  Similarly, we can show that $j^*$
is injective.  Thus we have the following.

\begin{lmm}
  \label{lmm:another definition of homotopy groups}
  For every $n \geq 0$ there is a natural isomorphism
  \[
    \pi_n(X,x_0) \cong [\bI^{n+1},e \semicolon X,x_0].
  \]
\end{lmm}

%
Suppose $f \colon (\bI^{n+1},e) \to (X,x_0)$ is a tame representative
of an element of $\pi_n(X,x_0)$, and $\ell \colon I \to X$ a tame path
from $x_0$ to $x_1$.
Then, by applying Theorem~\ref{thm:WHEP} to the tame homotopy
$e \times I \to X$ given by $l$, we obtain a homotopy $f \simeq g$
such that $g(e) = x_1$ holds.
Thus we can construct
\[
  \ell_{\sharp} \colon \pi_n(X,x_0) \to \pi_n(X,x_1)
\]
to be the map which takes $[f] \in \pi_n(X,x_0)$ to the class
$[g] \in \pi_n(X,x_1)$.

We leave it to the reader to verify the following.

\begin{prp}
  \label{prp:invariance under basepoint change}
  To every tame path $\ell \colon I \to X$ joining $x_0$ to $x_1$,
  there attached a group isomorphism
  $\ell_{\sharp} \colon \pi_n(X,A,x_0) \to \pi_n(X,x_1)$.
  If $\ell' \colon I \to X$ is another tame path joining $x_1$ to
  $x_2$ then we have
  $(\ell * \ell')_{\sharp} = \ell'_{\sharp} \circ \ell_{\sharp}$.
\end{prp}

As a final result of this section, we show that homotopy invariance
holds for certain type of adjunction spaces.

\begin{prp}
  \label{prp:adjunction inherits homotopy equivalence}
  Let $C$ be a cubical complex and $k \colon K \to C$ be the inclusion
  of a subcomplex.  Suppose $\phi \colon K \to Y$ is a tame map.  If
  $f \colon Y \to Z$ is a homotopy equivalence then so is the induced
  map $Y \cup_{(\phi,k)} C \to Z \cup_{(f\phi,k)} C$.
\end{prp}

We prove this by way of lemmas.  Given two smooth maps
$f \colon Y \to Z$ and $g \colon Y' \to Z$, we denote by $f \bigcup g$
the composition
\[
  \textstyle \nabla \circ (f \coprod g) \colon Y \coprod Y' \to Z
  \coprod Z \to Z,
\]
where $\nabla$ is the folding map of $Z \coprod Z$ onto $Z$.

\begin{lmm}
  \label{lmm:subduction of pushout}
  Let $\Phi \colon C \to Y$ be a smooth map and let
  $Z = Y \cup_{(\phi,k)} C$, where $\phi = \Phi|K\colon K \to Y$.  Let
  $i \colon Y \to Z$ be the inclusion.  Then the map
  \[
    \textstyle i \times 1 \bigcup \Phi \times 1 \colon Y \times I
    \coprod C \times I \to Z \times I
  \]
  is a subduction.
\end{lmm}

\begin{proof}
  Let $P \colon U \to Z \times I$ be a plot of $Z \times I$ given by
  $P(r) = (\sigma(r),\sigma'(r))$, and let $r \in U$.  Since $\sigma$
  is a plot of $Z$, there exists a plot $Q_1 \colon V \to Y$ such that
  $\sigma |V = i \circ Q_1$ holds, or a plot $Q_2 \colon V \to C$
  such that $\sigma|V = \Phi \circ Q_2$ holds.  In either case, we
  have
  \[
    \textstyle P|V = (i \times 1 \bigcup \Phi \times 1) \circ
    (Q_{\alpha},\sigma')|V,
  \]
  where $\alpha$ is either $1$ or $2$.  This means that
  $i \times 1 \bigcup \Phi \times 1$ is a subduction.
\end{proof}

Given a homotopy $h \colon K \times I \to Y$, its \emph{homotopy
  cylinder} $M(h)$ is defined to be the adjunction space
\[
  M(h) = (Y \times I) \cup_{(\tilde{h},k \times 1)} (C \times I)
\]
where $\tilde{h} \colon K \times I \to Y \times I$ %
is given by the formula, $\tilde{h}(v,t) = (h(v,t),t)$.  Let
\[
  \iota_{\alpha} \colon Y \cup_{(h_{\alpha},k)} C \to M(h) \quad
  (\alpha = 0,\, 1)
\]
be the inclusion induced by $Y \times \{\alpha\} \subset Y \times I$
and $C \times \{\alpha\} \subset C \times I$.

\begin{lmm}
  \label{lmm:retraction of homotopy cylinder}
  If $h \colon K \times I \to Y$ is a tame homotopy between $h_0$ and
  $h_1$ then there are deformation retractions
  $p_{\alpha} \colon M(h) \to Y \cup_{(h_{\alpha},k)} C \ (\alpha =
  0,\,1)$, and hence a homotopy equivalence
  $p_0 \circ \iota_1 \colon Y \cup_{(h_1,\,k)} C \simeq Y
  \cup_{(h_0,\,k)} C$.
\end{lmm}

\begin{proof}
  We assume that $C$ is a finite complex and show that a deformation
  retraction $p_0$ exists.  Clearly, this means that the lemma holds
  for all $C$ which is a disjoint union of finite complexes.
  As $h_0 \colon K \to Y$ is tame, there is a homotopy
  $g \colon C \times I \to C$ from the identity to a tame map
  $\rho \colon C \to C$ such that $h_0(g(v,t)) = h_0(v)$ holds for
  $(v,t) \in K \times I$ (cf.\ the proof of
  Proposition~\ref{prp:tameness of lifts}).
  But then, there exists by Theorem~\ref{thm:WHEP} a tame homotopy
  $H' \colon C \times I \to M(h)$ extending
  \[
    (\iota_0 \circ \Phi_0 \circ \rho) \cup (i \circ \tilde{h}) %
    \colon C \times \{0\} \cup K \times I \to M(h),
  \]
  where $\Phi_0$ is the natural map $C \to Y \cup_{(h_0,k)} C$ and
  $i \colon Y \times I \to M(h)$ is the inclusion.
  Now, let us choose $\epsilon > 0$ such that $H'$ is constant on
  $C \times [0,\epsilon]$ and define a smooth homotopy
  $H \colon C \times I \to M(h)$ by putting
  \[
    H(v,t) =
    \begin{cases}
      H'(v,t), & \epsilon/2 \leq t \leq 1
      \\
      \iota_0(\Phi_0(g(\lambda(4t/\epsilon)))), %
      & 0 \leq t \leq \epsilon/2.
    \end{cases}
  \]
  %
  Define $j \colon (Y \times I) \times I \to M(h)$ and
  $\Psi \colon (C \times I) \times I \to M(h)$ by putting
  \[
    j(y,t,s) = i(y,(1-s)t) \text{ and } \Psi(v,t,s) = H(v,(1-s)t),
  \]
  Then we have
  $j \circ (\tilde{h} \times 1) = \Psi \circ (k \times 1 \times 1)$.
  Hence we can apply Lemma~\ref{lmm:subduction of pushout} to obtain a
  homotopy $\tilde{H} \colon M(h) \times I \to M(h)$ such that
  $\tilde{H}_0$ is the identity and $\tilde{H}_1$ factors as a
  composition
  \[
    M(h) \xrightarrow{p_0} Y \cup_{(h_0,\,k)} C
    \xrightarrow{\iota_0} M(h).
  \]
  It follows that $p_0 \colon M(h) \to Y \cup_{(h_0,\,k)} C$ is a
  deformation retraction.
\end{proof}

The next lemma is inspired by the construction given in
\cite[Chapter~7]{Brown}.

\begin{lmm}
  \label{lmm:homotopy induces homotopy equivalence}
  Let $h_0$ and $h_1$ be tame maps from $K$ to $Y$, and let
  $[h_0 \semicolon h_1]$ be the set of homotopy classes of homotopies
  between $h_0$ and $h_1$ relative to end maps.  Then there is a
  correspondence
  \[
    \gamma_{h_0,h_1} \colon [h_0 \semicolon h_1] \to %
    [Y \cup_{(h_1,\,k)} C \semicolon Y \cup_{(h_0,\,k)} C]
  \]
  enjoying the following properties:
  \begin{enumerate}
  \item to any tame homotopy $h$ between $h_0$ and $h_1$ there exists
    a homotopy equivalence
    $\gamma \colon Y \cup_{(h_1,\,k)} C \to Y \cup_{(h_0,\,k)} C$
    satisfying $\gamma_{h_0,h_1}([h]) = [\gamma]$;
  \item if $h$ and $h'$ are respective homotopies $h_0 \simeq h_1$ and
    $h_1 \simeq h_2$ then we have
    $\gamma_{h_0,h_2}([h * h']) = \gamma_{h_0,h_1}([h]) \circ
    \gamma_{h_1,h_2}([h'])$;
  \item if $c_f$ is the constant homotopy of $f \colon K \to Y$ then
    $\gamma_{f,f}([c_f])$ is the class of the identity of
    $Y \cup_{(f,\,k)} C$;
  \item for any map $f \colon K \to Y$ and homotopy
    $g \colon Y \times I \to Y$ we have
    \[
      \gamma_{g_0 f,g_1 f}([h]) \circ [\bar{g}_1] = [\bar{g}_0]
    \]
    where $h$ is the composition
    $g (f \times 1) \colon K \times I \to Y$ and $\bar{g}_{\alpha}$ is
    the map $Y \cup_{(f,\,k)} C \to Y \cup_{(g_{\alpha} f,\,k)} C$
    induced by $g_{\alpha} \colon Y \to Y\ (\alpha = 0,\, 1)$.
  \end{enumerate}
\end{lmm}

\begin{proof}
  By applying Lemma~\ref{lmm:retraction of homotopy cylinder} to the
  homotopy cylinder $M(h)$, we can construct a homotopy equivalence
  \[
    \gamma(h) = p_0 \circ \iota_1 \colon Y \cup_{(h_1,g)} C \to Y
    \cup_{(h_0,g)} C.
  \]
  To see that $\gamma(h)$ depends only on the homotopy class of $h$,
  let us take another tame homotopy $h'$ from $h_0$ to $h_1$, and let
  $G$ be a homotopy $h \simeq h'$ relative to end maps.
  Let $\iota'_1 \colon Y \cup_{(h_1,k)} C \to M(h')$ and
  $p'_0 \colon M(h') \to Y \cup_{(h_0,k)} C$ denote, respectively, the
  inclusion and the retraction arising from $h'$.

  We need to show that $\gamma(h') = p'_0 \circ \iota'_1$ is homotopic
  to $\gamma(h)$.  Let
  \[
    W(G) = Y \times I \times I \cup_{(\bar{G},k \times 1 \times 1)} C
    \times I \times I,
  \]
  where $\tilde{G} \colon K \times I \times I \to Y \times I \times I$
  is given by $\tilde{G}(v,t,u) = (G(v,u,t),u,t)$.
  Since $\tilde{G}(v,0,u) = (\tilde{h}(v,u),0)$ and
  $\tilde{G}(v,1,u) = (\tilde{h}'(v,u),1)$ hold, there are inclusions
  $\eta \colon M(h) \to W(G)$ and $\eta' \colon M(h') \to W(G)$.
  It is now clear that there are two deformation retractions of $W(G)$
  onto $Y \cup_{(h_{\alpha},k)} C$, one factors through
  $p_0 \colon M(h) \to Y \cup_{(h_{\alpha},k)} C$ and another factors
  through $p'_0 \colon M(h') \to Y \cup_{(h_{\alpha},k)} C$.
  Evidently, this implies that $\gamma(h')$ is homotopic to
  $\gamma(h)$.
  Therefore, there exists a well-defined map
  \[
    \gamma_{h_0,h_1} \colon [h_0 \semicolon h_1] \to [Y \cup_{(h_1,k)}
    C \semicolon Y \cup_{(h_0,k)} C]
  \]
  given by $\gamma_{h_0,h_1}([h]) = [\gamma(h)]$, which by its
  definition satisfies the property (1).
  It is a routine task to verify the rest of the properties.
\end{proof}

\begin{proof}[Proof of Proposition~\ref{prp:adjunction inherits homotopy
    equivalence}]
  Let $g \colon W \to Z$ be a homotopy inverse to $f$.  Then we have a
  commutative diagram consisting of pushout squares:
  \[
    \xymatrix@C=16pt{%
      Z \ar[r]^-{f} \ar[d]^-{i} & W \ar[r]^-{g} \ar[d]^-{j} & Z
      \ar[r]^-{f} \ar[d]^-{i'} & W \ar[d]^-{j'}
      \\
      Z \cup_{(\phi,k)} C \ar[r]^-{\bar{f}} & W \cup_{(f\phi,k)} C
      \ar[r]^-{\bar{g}} & Z \cup_{(gf\phi,k)} C \ar[r]^-{\bar{f'}} & W
      \cup_{(fgf\phi,k)} C.  }%
  \]
  By the property (4) of Lemma~\ref{lmm:homotopy induces homotopy
    equivalence}, there exist homotopy equivalences
  \[
    \gamma \colon Z \cup_{(gf\phi,k)} C \to Z \cup_{(\phi,k)} C, \ \
    \gamma' \colon W \cup_{(fgf\phi,k)} C \to W \cup_{(f\phi,k)} C,
  \]
  induced by the respective homotopies $1 \simeq gf$ and
  $1 \simeq fg$, such that $\gamma \bar{g} \bar{f} \simeq 1$ and
  $\gamma' \bar{f}' \bar{g} \simeq 1$ hold.  But this means that
  $\bar{g}$ has a right homotopy inverse $\bar{f}\gamma$ and a left
  homotopy inverse $\gamma' \bar{f}'$.  Hence $\bar{g}$ is a homotopy
  equivalence and so is $\bar{f}$.
\end{proof}

\section{Model category of diffeological spaces}
\label{sec:model category}
In this section we shall show that the category $\Diff$ has a model
structure by arguing as in the proof of \cite[Proposition 8.3]{Spa}.

\begin{dfn}
  \label{dfn:admissibility}
  Suppose $K$ is a cubical subcomplex of $I^n$ and
  $0 < \epsilon \leq 1/2$.  A smooth map $f \colon K \to X$ is said to
  be \emph{$\epsilon$-admissible} if it is $\epsilon^{j+1}$-tame on
  each $j$-dimensional face of $K$.  
\end{dfn}

Clearly, $\epsilon$-tameness implies $\epsilon$-admissibility, and
conversely, $\epsilon$-admissibility implies
$\epsilon^{\dim{K}+1}$-tameness.
Since smooth maps are homotopic to tame maps, we can convert any
smooth map into an admissible one.  But in the sequel we require a
more sophisticated way of doing this.

\begin{prp}
  \label{prp:admissible replacement}
  Suppose $f \colon K \to X$ is $\epsilon$-admissible on a cubical
  subcomplex $L$ of $K$.  Then there is a homotopy $f \simeq g$
  relative to $L$ such that $g$ is $\epsilon$-admissible.
\end{prp}

\begin{proof}
  It is easy to construct a homotopy $f \simeq f' \rel{L}$ such that
  $f'$ is $\sigma$-tame if $\sigma < \epsilon^{\dim{L}+1}$ (cf.\
  Lemma~\ref{lmm:taming of maps}).  Hence we may assume from the
  beginning that $f$ is a tame map.
  For $0 \leq j \leq \dim K$, let $\bar{K}^j = L \cup K^j$ be the
  union of $L$ and the $j$-skeleton of $K$.
  Starting from the constant homotopy
  $\tilde{h}^0 \colon \bar{K}^0 \times I \to X$, we inductively
  construct a tame homotopy
  $\tilde{h}^j \colon \bar{K}^j \times I \to X$ from $f|\bar{K}^j$ to
  an $\epsilon$-admissible map $g^j$ relative to $L$.

  Suppose $\tilde{h}^{j-1}$ exists.  Let $F$ be a $j$-dimensional face
  not contained in $L$.  As $\partial F \subset \bar{K}^{j-1}$, there
  is a tame map
  \(
    h^F \colon (\partial F \times I) \cup (F \times \{0\}) \to X,
  \)
  which takes $(t,u)$ to $\tilde{h}^{j-1}(t,u)$ if $t \in \partial F$
  and to $f(t)$ if $u = 0$.
  But as $g^{j-1}|\partial F$ is $\epsilon^j$-tame and
  $(\partial F \times I) \cup (F \times \{0\})$ is linearly
  diffeomorphic to $J^{j-1}$, we can apply Proposition~\ref{prp:tame
    maps are extendable} with sufficiently small $\sigma$ and
  $\sigma' = \epsilon^{j+1}$ to obtain a tame extension
  $\tilde{h}^F \colon F \times I \to X$ such that
  $g^F = \tilde{h}^F|F \times \{1\} \colon F \to X$ is
  $\epsilon$-admissible.
  Thus, if we define $\tilde{h}^j \colon \bar{K}^j \times I \to X$ to
  be the union $\bigcup_F \tilde{h}^F$, where $F$ runs through
  $j$-dimensional faces of $K$ not contained in $L$, then
  $\tilde{h}^j$ gives a tame homotopy $f|\bar{K}^j \simeq g^j \rel{L}$
  such that $g^j = \bigcup_F g^F$ is $\epsilon$-admissible.  This
  proves the induction step.
\end{proof}

Especially, we have the following.

\begin{crl}
  \label{crl:admissible maps are extendable}
  Any $\epsilon$-admissible map $f \colon L^{n-1} \to X$ can be
  extended to an $\epsilon$-admissible map $I^n \to X$.
\end{crl}

\begin{dfn}
  \label{dfn:K-fibration}
  Let $\K$ be either $\I$ or $\J$.
  A smooth map $f \colon X \to Y$ is called a \emph{$\K$-fibration} if
  for every member $k_n \colon K^{n-1} \to I^n$ of $\K$
  and every pair of $\epsilon$-admissible maps
  $r \colon K^{n-1} \to X$ and $s \colon I^n \to Y$ satisfying
  $f \circ r = s \circ k_n$, there exists an $\epsilon$-admissible
  lift $h \colon I^n \to X$ which makes the two triangles in the
  diagram below commutative:
  \[
    \vcenter{%
      \xymatrix{%
        K^{n-1} \ar[d]^{k_n} \ar[r]^r & X \ar[d]^f
        \\
        I^n \ar[r]^s \ar@{.>}[ru]^{h} & Y. }}%
  \]
  
  A smooth map $f \colon X \to Y$ is called a \emph{$\K$-cofibration}
  if it has the left lifting property with respect to $\K$-fibrations,
  that is, for every commutative square
  \begin{equation*}
    \label{eqn:commutative square}
    \vcenter{%
      \xymatrix@C=32pt{%
        X \ar[d]^f \ar[r] & E \ar[d]^p
        \\
        Y \ar@{.>}[ur] \ar[r] & B }}%
  \end{equation*}
  such that $p \colon E \to B$ is a $\K$-fibration, there exists a
  lift $Y \to E$ making the two triangles commutative.

  A $\K$-fibration or a $\K$-cofibration $f \colon X \to Y$ is called
  to be \emph{trivial} if it is a weak homotopy equivalence, that is,
  the induced map $f_{\ast} \colon \pi_n(X,x) \to \pi_n(Y,f(x))$ is a
  bijection for every $x \in X$ and $n \geq 0$.
\end{dfn}

\begin{ex}
  The class of $\J$-fibrations contains many useful weak fibrations
  including the ones listed in Example~\ref{ex:weak_fibrations}.
  In fact, every constant map $X \to *$ is a $\J$-fibration by
  Corollary~\ref{crl:admissible maps are extendable}, and so is any
  diffeological fiber bundle.
  To see that $p \colon P(X,x_0) \to X$ is a $\J$-fibration, let us
  take an $\epsilon$-admissible pair $u \colon L^{n-1} \to P(X,x_0)$
  and $v \colon I^n \to X$.  
  Let $K = L^{n-1} \times I \cup I^n \times \{0,1\}$ and
  $u' \colon K \to X$ be a tame map which takes $(t,s) \in K$ to
  $u(t)(s)$ if $t \in L^{n-1}$, to $v(t)$ if $s = 0$, and to $x_0$ if
  $s = 1$.
  To obtain an $\epsilon$-admissible lift $I^n \to P(X,x_0)$ of $u$,
  it suffices to extend $u'$ to a tame map
  $\tilde{u} \colon I^n \times I \to X$ which is $\epsilon$-admissible
  with respect to the first $n$ coordinates.  We accomplish this by
  extending $u'$ in several steps.
  Let $A = I^n(\epsilon^{n+1}) \times I$,
  $B = I^{n-1}(\epsilon^{n+1}) \times [\epsilon^{n+1},1] \times I$,
  and $C = I^n - \Int{B}$.
  As we have $C = K \cup (\bar{L}^{n-1} \times I)$, where
  $\bar{L}^{n-1}$ is the closure of the $\epsilon^{n+1}$-neighborhood
  of $L^{n-1}$, and $v$ is $\epsilon^{n+1}$-tame, $u'$ can be extended
  to a tame map $\tilde{u}' \colon C \to X$ in an evident manner.
  But as $(A, A \cap C) \cong (I^{n+1},L^n)$ and $\tilde{u}'$ is tame
  on $A \cap C$, there exists an extension
  $\tilde{u}'' \colon A \to X$ of $\tilde{u}'|A \cap C$ having enough
  tameness on $A \cap (I^{n-1} \times \{1-\epsilon^{n+1}\} \times I)$
  (cf.\ Proposition~\ref{prp:tame maps are extendable}).
  It is now clear that $\tilde{u}''$ can be extended trivially to
  $\tilde{u}''' \colon B \to X$, and the resulting map
  $\tilde{u} = \tilde{u}' \cup \tilde{u}''' \colon I^{n+1} = C \cup B
  \to X$ extends $u'$ and is $\epsilon$-admissible with respect to the
  first $n$ coordinates.
\end{ex}

\begin{thm}
  \label{thm:model structure}
  The category $\Diff$ has a structure of a model category in the
  sense of \cite[Definition 3.3]{Spa}, where a smooth map
  $f \colon X \to Y$ is
  \begin{enumerate}
  \item a weak equivalence if $f$ is a weak homotopy equivalence,
  \item a fibration if $f$ is a $\J$-fibration, and
  \item a cofibration if $f$ is an $\I$-cofibration.
  \end{enumerate}
\end{thm}


We shall prove Theorem~\ref{thm:model structure} by verifying the
following axioms (cf.~\cite{Spa}).

\begin{itemize}
\item[\bf MC1] Finite limits and colimits exist.
\item[\bf MC2] If $f$ and $g$ are maps such that $g \circ f$ is
  defined and if two of the three maps $f$, $g$, $g \circ f$ are weak
  equivalences, then so is the third.
\item[\bf MC3] If $f$ is a retract of $g$ and $g$ is a fibration,
  cofibration, or a weak equivalence, then so is $f$.
\item[\bf MC4] Given a commutative square of the form
  \begin{equation*}
    \vcenter{%
      \xymatrix@C=32pt{%
        A \ar[r] \ar[d]^-{i} & X \ar[d]^-{p}
        \\
        B \ar[r] \ar@{.>}[ru] & Y }}%
  \end{equation*}
  the dotted arrow exists so as to make the two triangles commutative
  if either (i) $i$ is a cofibration and $p$ is a trivial fibration,
  or (ii) $i$ is a trivial cofibration and $p$ is a fibration.
\item[\bf MC5] Any map $f$ can be factored in two ways: (i)
  $f = p \circ i$, where $i$ is a cofibration and $p$ is a trivial
  fibration, and (ii) $f = p \circ i$, where $i$ is a trivial
  cofibration and $p$ is a fibration.
\end{itemize}

Axiom \textbf{MC1} follows from the fact that $\Diff$ has small limits
and colimits, 
and \textbf{MC2} follows from the functoriality of induced maps
combined with the change of basepoint homomorphism
(Proposition~\ref{prp:invariance under basepoint change}).  Axiom
\textbf{MC3} is straightforward from the definitions (cf.\
\cite[8.10]{Spa}).  In order to verify \textbf{MC4} and \textbf{MC5},
we need several lemmas and propositions.

We say that a diffeological space $X$ has the weak diffeology with
respect to its covering $\{X_j\}_{j \in J}$ if it satisfies the
following condition: a parameterization $P \colon U \to X$ is a plot
of $X$ if and only if there is an element $j \in J$ such that $P$ is
written locally as the composition of a plot of $X_j$ with the
inclusion $X_j \to X$.

\begin{lmm}\label{lmm:weak diffeology of subduction}
  Let $\delta$ be an ordinal and $X \colon \delta \to \Diff$ a
  $\delta$-sequence of inclusions.  Then $\colim X$ has
  the weak diffeology with respect to the covering consisting of the
  images of $X_{\alpha}\ (\alpha < \delta)$.
\end{lmm}

\begin{proof}
  Let $\pi \colon \coprod_{\alpha < \delta} X_{\alpha} \to \colim X$
  be the natural map.  Given a plot $P \colon U \to \colim X$ of
  $\colim X$ and $r$ in $U$, there exist an open neighborhood $V$ of
  $r$ and a plot $Q \colon V \to \coprod_{\alpha < \delta}X_{\alpha}$
  of $\coprod_{\alpha < \delta}X_{\alpha}$ such that
  $P|V = \pi \circ Q$ holds.  Then, by the definition of sum
  diffeology, there exist an open neighborhood $W$ of $r$ and a plot
  $Q^{\prime} \colon W \to X_{\beta}$ of $X_{\beta}$ such that $Q|W$
  is the composition of $Q'$ with the inclusion
  $i_{\beta} \colon X_{\alpha} \to \coprod_{\alpha <
    \delta}X_{\alpha}$.  Thus we have
  \[
    P|W= \pi \circ Q|W= \pi \circ i_{\beta} \circ Q^{\prime},
  \]
  showing that $\colim X$ has the weak diffeology with respect
  to the covering consisting of the images of
  $\pi \circ i_{\alpha} \colon X_{\alpha} \to \colim X$.
\end{proof}

\begin{prp}\label{prp:finite cell}
  The spaces $I^n,\ \bI^n$ and $L^{n-1}$ are finite relative to
  inclusions. 
\end{prp}

\begin{proof}
  Suppose $X \colon \delta \to \Diff$ is a $\delta$-sequence of
  inclusions.  We need only show that if $f \colon I^n \to \colim X$
  is smooth then its image is contained in some $X_{\alpha}$.  To see
  this, put $g = f \circ \lambda^n \colon \R^n \to \colim X$.  Then $g$
  is a plot of $\colim X$, and we have $f(I^n) = g(I^n)$ because
  $\lambda^n(I^n)$ covers the image $\lambda^n$.
  By Lemma \ref{lmm:weak diffeology of subduction}, there exist for
  any $v \in I^n$ an open neighborhood $V_v$ of $v$ and a plot
  $P_v \colon V_v \to X_{\alpha(v)}$ such that $g|V_v$ coincides
  with the composition of $P_v$ with the inclusion
  $X_{\alpha(v)} \to \colim X$.  Since
  $I^n \subset \cup_{v \in I^n}V_v$ and $I^n$ is compact, there
  exist $v_1,\, \cdots,\, v_k \in I^n$ satisfying
  \[
    \textstyle f(I^n) = g(I^n) \subset \bigcup_{1 \leq j \leq k}
    X_{\alpha(v_k)}.
  \]
  Thus by putting
  $\alpha = \mbox{max}\{\alpha(v_k) \mid 1 \leq j \leq k \}$ we have
  $f(I^n) \subset X_{\alpha}$.
\end{proof}

\begin{crl}
  \label{crl:weak homotopy equivalence}
  Let $X \colon \delta \to \Diff$ be a $\delta$-sequence of inclusions
  that are also weak equivalences.  Then the inclusion
  $X_0 \to \colim X$ is a weak equivalence.
\end{crl}

\begin{proof}
  Let $i_{\alpha} \colon X_{\alpha} \to \colim X$ and
  $j_{\alpha} \colon X_0 \to X_{\alpha}$ denote the natural maps.  We
  need to show that the induced map
  $i_{0 \ast} \colon \pi_n(X_0,x) \to \pi_n(\colim X,i_0(x))$
  is bijective for every $n \geq 0$ and $x \in X_0$.

  Suppose $n = 0$.  Then, for any $y \in \colim X$ there is an
  $\alpha$ such that $y$ is contained in $X_{\alpha}$.  But, as
  $j_{\alpha}$ is a weak equivalence, there exists a point $x \in X_0$
  contained in the path component of $y$.  Thus
  $i_{0 \ast} \colon \pi_0 X_0 \to \pi_0 \colim X$ is surjective.
  That $i_{0 \ast}$ is injective follows from the preceding corollary.

  Suppose now that $g$ is an element of $\pi_n(\colim X,i_0(x))$,
  where $n \geq 1$.  Then, by Proposition \ref{prp:finite cell}, there
  exists a map
  $\gamma \colon (I^n,\bI^n) \to (X_{\alpha},j_{\alpha}(x))$ such
  that $g = i_{\alpha *}([\gamma])$ holds.  Since
  $j_{\alpha} \colon X_0 \to X_{\alpha}$ is a weak equivalence,
  there is an element $[\tilde{\gamma}] \in \pi_n(X_0,x)$
  satisfying $j_{\alpha \ast}([\tilde{\gamma}]) = [\gamma]$.  Thus we
  have
  \[
    g = i_{\alpha \ast}([\gamma]) %
    = i_{\alpha \ast}(j_{\alpha \ast}([\tilde{\gamma}])) %
    = i_{\alpha \ast} j_{\alpha \ast}([\tilde{\gamma}]) %
    = i_{0\ast}([\tilde{\gamma}])
  \]
  implying that $i_{0 \ast}$ is surjective.  That $i_{0 \ast}$ is
  injective is proved similarly.
\end{proof}

\begin{prp}
  \label{prp:trivial J-fibration is I-fibration}
  Let $p \colon X \to Y$ be a smooth map between diffeological spaces.
  Then $p$ is an $\I$-fibration if and only if it is a trivial
  $\J$-fibration.
\end{prp}

\begin{proof}
  We first show that an $\I$-fibration $p \colon X \to Y$ is a weak
  equivalence, that is, the induced map
  $p_{\ast}\colon \pi_n(X,x) \to \pi_n(Y,p(x))$ is bijective for
  every $n \geq 0$ and $x \in X$.  Let
  $\gamma \colon (I^n, \bI^n) \to (Y,p(x))$ be a tame map,
  and let $c_x \colon \bI^n \to X$ be the constant map with value
  $x \in X$.  Then we have a commutative square
  \[
    \xymatrix@C=32pt{%
      \bI^n \ar[r]^-{c_x} \ar[d]^-{i_n} & X \ar[d]^-{p}
      \\
      I^n \ar[r]^-{\gamma} & Y.
    }%
  \]
  Since $(c_x,\gamma)$ is $\epsilon$-admissible for some
  $\epsilon > 0$, there is a lift $\tilde{\gamma} \colon I^n \to X$
  satisfying $\tilde{\gamma} \circ i_n = c_x$ and
  $p \circ \tilde{\gamma} = \gamma$.  Thus we have
  $p_{\ast}([ \tilde{\gamma}]) = [\gamma]$, implying that $p_{\ast}$
  is a surjection.
  To see that $p_*$ is injective, let $\gamma_0$ and $\gamma_1$ be
  tame maps $(I^n,\bI^n) \to (X,x)$ such that
  $p_{\ast}([\gamma_0]) = p_{\ast}([\gamma_1])$ holds in
  $\pi_n(Y,p(x))$.
  Then there exists a tame homotopy $H \colon I^n \times I \to Y$
  between $p \circ \gamma_0$ and $p \circ \gamma_1$ relative to
  $\bI^n$.
  Let $\gamma \colon \bI^{n+1} \to X$ be a tame map which takes
  $(t,s)$ to $\gamma_s(t)$ if
  $(t,s) \in I^n \times \{0,1\}$, and to $x$ if
  $(t,s) \in \bI^n \times I$.  Then we have a commutative square
  \[
    \xymatrix{%
      \bI^{n+1} \ar[r]^-{\gamma} \ar[d]^-{i_{n+1}}& X \ar[d]^-{p}
      \\
      I^n \times I \ar[r]^-{H} & Y
    }%
  \]
  Hence there exists a lift $\tilde{H} \colon I^n \times I \to X$
  which gives a homotopy $\gamma_0 \simeq \gamma_1$.  Thus we have
  $[\gamma_0] = [\gamma_1]$, showing that $p_{\ast}$ is injective.

  To see that $p$ is a $\J$-fibration, take an $\epsilon$-admissible
  pair consisting of $f \colon L^{n-1} \to X$ and
  $g \colon I^n \to Y$.  Then we have a commutative square
  \[
    \xymatrix@C=60pt{%
      \bI^{n-1} \times \{1\} \ar[d]^{i_{n-1}} \ar[r]^-{f|\bI^{n-1}
        \times \{1\}} & X \ar[d]^p
      \\
      I^{n-1} \times \{1\} \ar[r]^-{g|I^{n-1}\times \{1\}} & Y.  }%
  \]
  Since $f|\bI^{n-1} \times \{1\}$ and $g|I^{n-1}\times \{1\}$ are
  $\epsilon$-admissible, there is an $\epsilon$-admissible lift
  $\tilde{g} \colon I^{n-1} \times \{1\} \to X$, and consequently we
  can define $\tilde{f} \colon \bI^n \to X$ to be the union
  $f \cup \tilde{g} \colon L^{n-1} \cup I^{n-1} \times \{1\} \to X$.
  Clearly, $\tilde{f}$ is $\epsilon$-admissible, and hence there
  exists an $\epsilon$-admissible lift $G \colon I^n \to X$ satisfying
  $p \circ G = g$ and $G \circ i_n = \tilde{f}$.  But this means
  $G|L^{n-1} = f$, implying that $p$ is a $\J$-fibration.

  Conversely, suppose $p \colon X \to Y$ is a trivial $\J$-fibration.
  Let $f \colon \bI^n \to X$ and $g \colon I^n \to Y$ be an
  $\epsilon$-admissible pair.  We need to show that there is an
  $\epsilon$-admissible lift $G \colon I^n \to X$ satisfying
  $p \circ G = g$ and $G \circ i_n = f$.
  Let $e = (1,\cdots,1) \in I^n$ and $x = f(e)$.  Since
  $p \circ f = g|\bI^n$ is null homotopic and $p$ is a weak
  equivalence, there exists by Lemma~\ref{lmm:another definition of
    homotopy groups} a tame homotopy
  $F \colon (\bI^n,\{e\}) \times I \to (X,x)$ from the constant map to
  $f$.
  Here, $F$ can be taken to be $\epsilon$-admissible by
  Proposition~\ref{prp:admissible replacement}.
  Let us define $H \colon J^n \to Y$ by
  \[
  H(t,s) =
  \begin{cases}
    g(t), & (t,s) \in I^n \times \{1\}.
    \\
    p(F(t,s)), & (t,s) \in \bI^n \times I
  \end{cases}
  \]
  Since $H$ is $\epsilon$-admissible, it can be extended by
  Corollary~\ref{crl:admissible maps are extendable} to an
  $\epsilon$-admissible homotopy $H' \colon I^n \times I \to Y$ from
  $\gamma' \colon (I^n,\bI^n) \to (Y,p(x))$ to $g$.
  But as $p$ is a weak equivalence, there exist a tame map
  $\gamma \colon (I^n, \bI^n) \to (X,x)$ and a tame homotopy
  $H'' \colon (I^n,\bI^n) \times I \to (Y,p(x))$ from $p \circ \gamma$
  to $\gamma'$.
  Again, we may assume by Proposition~\ref{prp:admissible replacement}
  that $\gamma$ and $H''$ are $\epsilon$-admissible.
  Let $\tau$ be such that $\epsilon^2 < \tau < 1/2$, and define a
  bijection $\theta \colon [1/2,1] \to [0,1]$ by
  \[
    \theta(s) = T_{\epsilon^2,\tau}(s) - 1/2 +
    T_{1/2-\tau,1/2-\epsilon^2}(s-1/2).
  \]
  Then we can define $K \colon L^n \to X$ and
  $G' \colon I^n \times I \to Y$ by the formula,
  \begin{align*}
    K(t,s)
    &=
      \begin{cases}
        F(t,\theta(s)), & (t,s) \in\bI^n \times [1/2,1]
        \\
        x, & (t,s) \in \bI^n \times [0,1/2]
        \\
        \gamma(t), & (t,s) \in I^n \times \{0\},
      \end{cases}
    \\                 
    G'(t,s)
    &=
      \begin{cases}
        H'(t,\theta(s)), & 1/2 \leq s \leq 1.
        \\
        H''(t,1 - \theta(1-s)), & 0 \leq s \leq 1/2.
      \end{cases}
  \end{align*}
  As we have $1 - \theta(1-s) = s$ for $s \leq \epsilon^2$ and
  $\theta(s) = s$ for $s \geq 1-\epsilon^2$, both $K$ and $G'$ are
  $\epsilon$-admissible and $p \circ K = G' \circ j_{n+1}$ holds.
  Hence there is an $\epsilon$-admissible lift
  $\widetilde{G} \colon I^n \times I \to X$ satisfying
  $p \circ \widetilde{G} = G'$ and $\widetilde{G}|J^n = K$, and
  consequently, we have an $\epsilon$-admissible lift
  $G = \widetilde{G}|I^n \times \{1\} \colon I^n \to X$ satisfying
  $G \circ i_n = f$ and $p \circ G = g$.
\end{proof}

Proposition~\ref{prp:trivial J-fibration is I-fibration} says that $f$
is a cofibration if and only if it satisfies the left lifting property
with respect to trivial $\J$-fibrations.  Hence we have

\begin{crl}
  \label{crl:MC4 (ii)}
  Axiom\/ {\bf MC4} holds under the condition (i).
\end{crl}

We now introduce \emph{infinite gluing construction} similar to that
of \cite{Spa} by employing the terminology of \emph{relative $\K$-cell
  complex}.

\begin{dfn}
  \label{dfn:I-cell complex}
  Let $\K = \{k_n \colon K^{n-1} \to I^n\}$ be either $\I$ or $\J$.
  A smooth map $f \colon X \to Y$ is called a \emph{relative $\K$-cell
    complex} if there is an ordinal $\delta$ and a $\delta$-sequence
  $Z \colon \delta \to \Diff$ such that $f \colon X \to Y$ coincides
  with the composition $Z_0 \to \colim Z$ and that for each $\beta$
  such that $\beta + 1 < \delta$, there is a pushout square
  \[
    \vcenter{%
      \xymatrix{%
        \bI^{d_{\beta + 1}} \ar[d]^{k_{d_{\beta + 1}}}
        \ar[r]^-{\phi_{\beta + 1}}
        & Z_{\beta} \ar[d]
        \\
        I^{d_{\beta + 1}} \ar[r]^-{\Phi_{\beta + 1}} & Z_{\beta+1},
      }%
    }%
  \]
  that is, $Z_{\beta+1}$ is an adjunction space
  $Z_{\beta} \cup_{(\phi_{\beta + 1},\,k_{d_{\beta + 1}})} I^{d_{\beta
      + 1}}$ for some $d_{\beta + 1} \geq 0$.
  %
\end{dfn}

\begin{prp}
  \label{prp:factorization by infinite gluing}
  Given $f \colon X \to Y$, there is a commutative diagram
  \[
    \xymatrix@C=20pt@R=20pt{%
      X \ar@{=}[d] \ar[r] %
      & F^1(\K,f) \ar@<0.5ex>[d]^{\tau_1} \ar[r] %
      & F^2(\K,f) \ar@<0.5ex>[d]^{\tau_2} \ar[r] %
      & \cdots \ar[r] %
      & F^l(\K,f) \ar@<0.5ex>[d]^{\tau_l} \ar[r] %
      & \cdots %
      \\ %
      X \ar[d]^f \ar[r] %
      & G^1(\K,f) \ar@<0.5ex>[u]^{\sigma_1} \ar[d]^{p_1} \ar[r] %
      & G^2(\K,f) \ar@<0.5ex>[u]^{\sigma_2} \ar[d]^{p_2} \ar[r] %
      & \cdots \ar[r] %
      & G^l(\K,f) \ar@<0.5ex>[u]^{\sigma_l} \ar[d]^{p_l} \ar[r] %
      & \cdots %
      \\ %
      Y \ar@{=}[r] %
      & Y \ar@{=}[r] %
      & Y \ar@{=}[r] %
      & \cdots \ar@{=}[r] %
      & Y \ar@{=}[r] %
      & \cdots %
    }%
  \]
  such that the following hold:
  \begin{enumerate}
  \item Each arrow in the upper row is a relative $\K$-cell complex,
    and hence so is the composition $i'_l \colon X \to F^l(\K,f)$ for
    every $l \geq 1$.
  \item Each arrow in the middle row is an inclusion, and hence so is
    the composition $i_l \colon X \to G^l(\K,f)$ for every $l \geq 1$.
  \item The maps $\tau_l$ and $\sigma_l$ are homotopy inverse to each
    other.
  \item The colimit $i_{\infty} \colon X \to G^{\infty}(\K,f)$ of the
    maps $i_l$ is a $\K$-cofibration.
  \item The map $p_{\infty} \colon G^{\infty}(\K,f) \to Y$ induced by
    the maps $p_l$ is a $\K$-fibration.
  \end{enumerate}
\end{prp}

To prove this, we need the following lemma.

\begin{lmm}
  \label{lmm:universal admissible space}
  For any $\epsilon$ satisfying $0 < \epsilon \leq 1/2$, there is a
  subduction $\rho_{\epsilon} \colon I^n \to \tI^n_{\epsilon}$
  enjoying the following properties:
  Any smooth map $f \colon I^n \to X$ which is $\epsilon$-admissible
  on $K^{n-1}$ is $\epsilon$-admissible on $I^n$ if and only if it
  factors through $\rho_{\epsilon}$.
  Moreover, $\rho_{\epsilon}$ has a homotopy inverse
  $\iota_{\epsilon} \colon \tI^n_{\epsilon} \to I^n$ such that
  $\iota_{\epsilon} \circ \rho_{\epsilon} \simeq 1
  \rel{I^n(\epsilon^n)}$ and
  $\rho_{\epsilon} \circ \iota_{\epsilon} \simeq 1
  \rel{\rho_{\epsilon}(I^n(\epsilon^n))}$ hold.
\end{lmm}

\begin{proof}
  Define a smooth map $\rho_{\epsilon} \colon I^n \to I^n$ by the
  formula:
  \[
    \rho_{\epsilon}(t_1,\,\cdots\,,t_{n-1},u) =
    (T_{a(u),b(u)}(t_1),\,\cdots\,,T_{a(u),b(u)}(t_{n-1}),
    T_{\epsilon^{n+1},\epsilon^n}(u)),
  \]
  where $(a(u),b(u)) \equiv (\epsilon^{n+1},\epsilon^n)$ if
  $K^{n-1} = \bI^n$, while in the case of $K^{n-1} = L^{n-1}$, $a(u)$
  and $b(u)$ are taken to be non-decreasing functions satisfying
  $(a(u),b(u)) = (\epsilon^{n+1},\epsilon^n)$ for
  $0 \leq u \leq 1 - \epsilon^{n+1}$ and
  $(a(1),b(1)) = (\epsilon^n,\epsilon^{n-1})$ (cf.\ the proof of
  Proposition~\ref{prp:tame maps are extendable}).
  Let $\tI^n_{\epsilon}$ be the $n$-cube equipped with the pushforward
  by $\rho_{\epsilon}$ of the standard diffeology of $I^n$.  Then it
  follows by definition that
  $\rho_{\epsilon} \colon I^n \to \tI^n_{\epsilon}$ is a subduction.

  Suppose $f \colon I^n \to X$ is $\epsilon$-admissible on $K^{n-1}$.
  If $f$ factors through $\rho_{\epsilon}$ then $f$ is
  $\epsilon$-admissible on $I^n$ because $\rho_{\epsilon}$ is
  $\epsilon^{n+1}$-tame on $I^n$ and, additionally, is
  $\epsilon^n$-tame on $I^{n-1} \times \{1\}$ when $K^{n-1} = L^n$.
  On the other hand, if $f$ is $\epsilon$-admissible then we can
  define $\tilde{f} \colon \tI^n_{\epsilon} \to X$ by putting
  $\tilde{f}(t) = f(\rho_{\epsilon}^{-1}(t))$, which is well-defined
  because $\rho_{\epsilon}$ is bijective on $I^n(\epsilon^{n+1})$ and
  $f$ is $\epsilon^{n+1}$-tame.
  Thus we have a factorization $f = \tilde{f} \circ \rho_{\epsilon}$,
  which in turn implies $\tilde{f}$ is smooth because
  $\rho_{\epsilon}$ is a subduction.

  Now, let $\iota_{\epsilon} \colon \tI^n_{\epsilon} \to I^n$ be the
  identity map.  Then $\iota_{\epsilon}$ is smooth because the
  diffeology of $\tI^n_{\epsilon}$ is finer than the standard
  diffeology.
  Moreover, if we put
  \[
    h(t,u) = (1-u)\rho_{\epsilon}(t) + ut,\ \ (t,u) \in I^n \times I
  \]
  then $h \colon I^n \times I \to I^n$ gives
  $\iota_{\epsilon} \circ \rho_{\epsilon} \simeq 1$ rel
  $I^n(\epsilon^n)$.
  While on the other hand, if we define
  $\tilde{h} \colon \tI^n_{\epsilon} \times I \to \tI^n_{\epsilon}$ by
  the formula
  $\tilde{h}(t,u) = \rho_{\epsilon}(h(\rho_{\epsilon}^{-1}(t),u))$
  then $\tilde{h}$ is smooth because
  $\rho_{\epsilon} \circ h = \tilde{h} \circ (\rho_{\epsilon} \times
  1)$ holds and $\rho_{\epsilon} \times 1$ is a subduction.
  Clearly, $\tilde{h}$ gives a homotopy
  $\rho_{\epsilon} \circ \iota_{\epsilon} \simeq 1$ relative to
  $\rho_{\epsilon}(I^n(\epsilon^n))$.
\end{proof}

\begin{proof}[Proof of Proposition~\ref{prp:factorization by infinite
    gluing}]
  We proceed by induction on $l$.
  For $l = 0$, we take $F^0(\K,f) = G^0(\K,f) = X$,
  $i_0 = i'_0 = \sigma_0 = \tau_0 = 1$, and $p_0 = p'_0 = f$.
  Suppose $F^{l-1}(\K,f)$ and $G^{l-1}(\K,f)$ exist, and let
  $S(k_n,p_{l-1})$ be the set of pairs of tame maps
  $u \colon K^{n-1} \to G^{l-1}(\K,f)$ and $v \colon I^n \to Y$
  satisfying $p_{l-1} \circ u = v \circ k_n$.
  Let
  \[
    \textstyle
    C = \coprod_{n \geq 0}\coprod_{(u,v) \in S(k_n,p_{l-1})}I^n, \ \
    K = \coprod_{n \geq 0}\coprod_{(u,v) \in S(k_n,p_{l-1})}K^{n-1},
  \]
  and put
  \(
    F^l(\K,f) = F^{l-1}(\K,f) \cup_{(\sigma_{l-1}\bm{u},\,\bm{k})} C,
  \)
  where $\bm{k} \colon K \to C$ is the inclusion and
  $\bm{u} \colon K \to G^{l-1}(\K,f)$ is the union of the maps
  $u \colon K^{n-1} \to G^{l-1}(\K,f)$ for $(u,v) \in S(k_n,p_{l-1})$.
  More explicitly, $F^l(\K,f)$ is the union of adjunction spaces
  $F^{l-1}(\K,f) \cup_{(\sigma_{l-1}u,\,k_n)} I^n$.  Thus, with a
  suitable choice of total order on $\coprod_{n \geq 0}S(k_n,p_{l-1})$
  the inclusion $F^{l-1}(\K,f) \to F^l(\K,f)$ is a relative $\K$-cell
  complex, and hence so is the composition
  $i'_l \colon X \to F^l(\K,f)$.

  To construct $G^l(\K,f)$, choose for each $(u,v) \in S(k_n,p_{l-1})$
  the largest constant $\epsilon$ such that both $u$ and $v$ are
  $\epsilon$-admissible ($0 < \epsilon \leq 1/2$), and denote by
  $\rho_{u,v} \colon I^n \to \tI^n_{u,v}$ the subduction
  $\rho_{\epsilon} \colon I^n \to \tI^n_{\epsilon}$.
  Let $\tilde{K}^{n-1}_{u,v} = \rho_{u,v}(K^{n-1})$ and
  $\tilde{k}_n \colon \tilde{K}^{n-1}_{u,v} \to \tI^n_{u,v}$ be the
  inclusion.
  Since $u$ and $v$ are $\epsilon$-admissible, there exist smooth maps
  $\tilde{u} \colon \tilde{K}^{n-1}_{u,v} \to G^{l-1}(\K,f)$ and
  $\tilde{v} \colon \tI^n_{u,v} \to Y$ such that
  $u = \tilde{u} \circ \rho_{u,v}$ and
  $v = \tilde{v} \circ \rho_{u,v}$ hold (Lemma~\ref{lmm:universal
    admissible space}).
  Now, let
  \[
    \textstyle \tilde{C} = %
    \coprod_{n \geq 0}\coprod_{(u,v) \in S(k_n,p_{l-1})}\tI^n_{u,v},
    \ \ %
    \tilde{K} = \coprod_{n \geq 0}\coprod_{(u,v) \in S(k_n,p_{l-1})}
    \tilde{K}^{n-1}_{u,v},
  \]
  and put
  $G^l(\K,f) = G^{l-1}(\K,f) \cup_{(\bm{\tilde{u}},\,\bm{\tilde{k}})}
  \tilde{C}$, that is, the union of adjunction spaces
  $G^{l-1}(\K,f) \cup_{(\tilde{u},\,\tilde{k}_n)} \tI^n_{u,v}$.
  %
  As we have $p_{l-1} \circ \tilde{u} = \tilde{v} \circ \tilde{k}_n$,
  there is a map
  $(p_l)_{u,v} \colon G^{l-1}(\K,f) \cup_{(\tilde{u},\,\tilde{k}_n)}
  \tI^n_{u,v} \to Y$ making the diagram below commutative:
  \begin{equation}
    \label{diagram:adding a cell to G}
    \vcenter{%
      \xymatrix{%
        K^{n-1} \ar[d]^-{k_n} \ar[r]^-{\rho_{u,v}}
        & \tilde{K}^{n-1}_{u,v} \ar[d]^-{\tilde{k}_n}
        \ar[r]^-{\tilde{u}}
        & G^{l-1}(\K,f) \ar[d]^-{\subset} \ar[r]^-{p_{l-1}}
        & Y \ar@{=}[d]
        \\
        I^n \ar[r]^-{\rho_{u,v}}
        & \tI^n_{u,v} \ar[r]^-{\tilde{\Phi}_{u,v}}
        \ar@/_16pt/[rr]_-{\tilde{v}}
        & G^{l-1}(\K,f) \cup_{(\tilde{u},\,\tilde{k}_n)}
        \tI^n_{u,v} \ar[r]^-{(p_l)_{u,v}} & Y.
      }
    }%
  \end{equation}
  Thus we have an inclusion $G^{l-1}(\K,f) \to G^l(\K,f)$ and
  $p_l \colon G^l(\K,f) \to Y$ defined as the union of the maps
  $(p_l)_{u,v}$ above.

  We define $\tau_l \colon F^l(\K,f) \to G^l(\K,f)$ and
  $\sigma_l \colon G^l(\K,f) \to F^l(\K,f)$ to be the respective
  compositions $\tau''_l \circ \tau'_l$ and
  $\sigma'_l \circ \sigma''_l$ in the diagram below.
  For brevity, we abbreviate $F^{l-1}(\K,f)$ to $F^{l-1}$ and
  $G^{l-1}(\K,f)$ to $G^{l-1}$.
  \begin{equation*}
    \vcenter{%
      \xymatrix@C=24pt{%
        F^{l-1} \cup_{(\sigma_{l-1}\bm{u},\,\bm{k})} C
        \ar@<0.5ex>[r]^-{\tau'_l}
        & G^{l-1} \cup_{(\bm{u},\,\bm{k})} C
        \ar@<0.5ex>[l]^-{\sigma'_l} \ar@<0.5ex>[r]^-{\tau''_l}
        & G^{l-1} \cup_{(\bm{\tilde{u}},\,\bm{\tilde{k}})} \tilde{C}
        \ar@<0.5ex>[l]^-{\sigma''_l} 
      }
    }%
  \end{equation*}
  Here $\sigma'_l$ is the map induced by $\sigma_{l-1}$ and $\tau'_l$
  is the composition
  \[
    F^{l-1} \cup_{(\sigma_{l-1}\bm{u},\,\bm{k})} C
    \xrightarrow{\overline{\tau_{l-1}}} %
    G^{l-1} \cup_{(\tau_{l-1}\sigma_{l-1}\bm{u},\,\bm{k})} C
    \xrightarrow{\gamma} %
    G^{l-1} \cup_{(\bm{u},\,\bm{k})} C,
  \]
  where $\overline{\tau_{l-1}}$ is induced by $\tau_{l-1}$ and
  $\gamma$ is a homotopy equivalence given by
  Lemma~\ref{lmm:retraction of homotopy cylinder}.  As we have shown
  in the proof of Proposition~\ref{prp:adjunction inherits homotopy
    equivalence}, $\sigma'_l$ and $\tau'_l$ are homotopy inverse to
  each other.
  On the other hand, $\tau''_l$ and $\iota''_l$ are induced by the
  subductions $\rho_{u,v} \colon I^n \to \tI^n_{u,v}$ and their
  homotopy inverses $\iota_{u,v} \colon \tI^n_{u,v} \to I^n$ given by
  Lemma~\ref{lmm:universal admissible space}.
  But if $u$ is $\epsilon$-admissible then $u$ is $\epsilon^n$-tame,
  and hence the homotopies
  $\iota_{u,v} \circ \rho_{u,v} \simeq 1 \rel{I^n(\epsilon^n)}$ and
  $\rho_{u,v} \circ \iota_{u,v} \simeq 1
  \rel{\rho_{u,v}(I^n(\epsilon^n))}$ induce homotopy equivalence
  $G^{l-1}(\K,f) \cup_{(u,\,k_n)} I^n \simeq G^{l-1}(\K,f)
  \cup_{(\tilde{u},\,\tilde{k}_n)} \tI^n_{u,v}$ relative to
  $G^{l-1}(\K,f)$.
  Thus $\tau''_l$ and $\sigma''_l$ are homotopy inverse to each other,
  and so are the compositions $\tau_l$ and $\sigma_l$.

  We next show that $i_{\infty} \colon X \to G^{\infty}(\K,f)$ is a
  $\K$-cofibration.
  Let $q \colon E \to B$ be a $\K$-fibration, and let
  $r \colon X \to E$ and $s \colon G^{\infty}(\K,f) \to B$ be smooth
  maps satisfying $q \circ r = s \circ i_{\infty}$.
  Let us write $s_l = s|G^l(\K,f)$ for $l \geq 1$.
  Starting from $h_0 = r$, we shall inductively construct a lift
  $h_l \colon G^l(\K,f) \to E$ 
  such that the following diagram commutes.
  \begin{equation}
    \label{diagram:LLP for G}
    \vcenter{%
      \xymatrix{%
        X \ar[d]^-{i_l} \ar[r]^-{r} & E \ar[d]^-{q}
        \\
        G^l(\K,f) \ar[r]^-{s_l} \ar[ru]^{h_l} & B.}}%
  \end{equation}
  Suppose $h_{l-1}$ exists.  Then for every $(u,v) \in S(k_n,p_{l-1})$
  we have a commutative diagram (in which we abbreviate $G^l(\K,f)$ to
  $G^l$)
  \begin{equation}
    \label{diagram:i_infty is a cofibration}
    \vcenter{%
      \xymatrix{%
        K^{n-1} \ar[d]^-{k_n} \ar[r]^-{\rho_{u,v}} %
        & \tilde{K}^{n-1}_{u,v} \ar[d]^-{\tilde{k}_n}
        \ar[r]^-{\tilde{u}} & G^{l-1} \ar[d] \ar@{=}[r]
        & G^{l-1} \ar[d] \ar[r]^-{h_{l-1}}
        & E \ar[d]^-{q}
        \\
        I^n \ar[r]^-{\rho_{u,v}} \ar@{.>}[urrrr] %
        & \tI^n_{u,v} \ar[r]^-{\tilde{\Phi}_{u,v}}
        & G^{l-1} \cup_{(\tilde{u},\,\tilde{k}_n)} 
        \tI^n_{u,v} \ar[r]^-{\subset} 
        & G^l \ar[r]^-{s_l} & B.
      }%
    }%
  \end{equation}
  Since the upper and lower horizontal compositions are
  $\epsilon$-admissible, there is an $\epsilon$-admissible lift
  $h'_{u,v} \colon I^n \to E$ making the diagram commutative.
  But then $h'_{u,v}$ induces a unique map
  $h''_{u,v} \colon \tI_{u,v} \to E$ such that
  $h'_{u,v} = h''_{u,v} \circ \rho_{u,v}$, and hence
  $h''_{u,v} \circ \tilde{k}_n = h_{l-1} \circ \tilde{u}$ hold.
  Consequently, there exists a lift
  $\tilde{h}_{u,v} \colon G^{l-1}(\K,f)
  \cup_{(\tilde{u},\,\tilde{k}_n)} \tI^n_{u,v} \to E$ for every
  $(u,v) \in S(k_n,p_{l-1})$ such that
  \[
    h_l = \textstyle\bigcup_{n \geq 0} \bigcup_{(u,v) \in
      S(k_n,p_{l-1})} \tilde{h}_{u,v} \colon G^l(\K,f) \to E
  \]
  makes the diagram (\ref{diagram:LLP for G}) commutative.
  Now, by taking the colimit as $l \to \infty$ we obtain a lift
  $h_{\infty} \colon G^{\infty} \to E$ satisfying
  $h_{\infty} \circ i_{\infty} = r$ and $q \circ h_{\infty} = s$.

  To see that $p_{\infty}$ is a $\K$-fibration, let
  $r \colon K^{n-1} \to G^{\infty}(\K,f)$ and $s \colon I^n \to Y$ be
  an $\epsilon$-admissible pair, and $l$ be an integer such that the
  image of $r$ is contained in $G^l(\K,f)$
  (Proposition~\ref{prp:finite cell}).  Then there is a commutative
  diagram
  \begin{equation}
    \label{diagram:p_infty is a fibration}
    \vcenter{%
      \xymatrix{%
        K^{n-1} \ar[d]^-{k_n} \ar[r]^-{r} & G^l(\K,f)
        \ar[d]^-{p_l} \ar[r]^-{i_{l+1}} & G^{l+1}(\K,f)
        \ar[d]^-{p_{l+1}} \ar[r] & G^{\infty}(\K,f)
        \ar[d]^-{p_{\infty}}
        \\
        I^n \ar[r]^-{s} \ar@{.>}[urr] & Y \ar@{=}[r] & Y \ar@{=}[r]
        & Y.
      }%
    }%
  \end{equation}
  As $(r,s) \in S(k_n,p_l)$, there is an $\epsilon$-admissible lift
  $I^n \to G^{\infty}(\K,f)$ given as the composition
  \[
    I^n \xrightarrow{\rho_{r,s}} \tI^n_{r,s}
    \xrightarrow{\tilde{\Phi}_{r,s}} %
    G^l(\K,f) \cup_{(\tilde{r},\,\tilde{k}_n)}\tI^n_{r,s}
    \xrightarrow{\subset} G^{l+1}(\K,f) \to G^{\infty}(\K,f).
  \]
  This means $p_{\infty}$ is a $\K$-fibration, completing the proof
  the proposition.
\end{proof}

In the case $\K = \J$ the following additional property hold.

\begin{lmm}
  \label{lmm:deformation retract of cell}
  For any $f \colon X \to Y$, the map
  $i_{\infty} \colon X \to G^{\infty}(\J,f)$ is a weak equivalence.
\end{lmm}

\begin{proof}
  By Corollary~\ref{crl:weak homotopy equivalence}, it suffices to
  show that the inclusion $G^k(\J,f) \to G^{k+1}(\J,f)$ is a weak
  equivalence for every $k \geq 0$, or equivalently, so is
  $F^k(\J,f) \to F^{k+1}(\J,f)$.
  But this follows from the fact that if $u \colon L^{n-1} \to W$ is
  tame then there is a deformation retraction of the adjunction space
  $Z = W \cup_{(u,\,j_n)} I^n$ onto $W$.
  %
  To verify this, let $i \colon W \to Z$ be the
  inclusion and $\Phi \colon I^n \to Z$ be the
  natural map.  Suppose $u$ is $\epsilon$-tame for some
  $\epsilon > 0$, and define $H \colon I^n \times I \to I^n$ by
  \[
  H(t,s) = (1-s)t + sR_{\epsilon}(t),
  \]
  where $R_{\epsilon} \colon I^n \to L^{n-1}$ is an
  $\epsilon$-approximate retraction (cf.\ Lemma~\ref{lmm:approximate
    retraction exists}).
  Then the composition $F = \Phi \circ H \colon I^n \times I \to Z$
  satisfies
  \[
    F(t,0) = \Phi(t),\ \ F(t,1) \subset W,\ \ F(t,s) = i(u(t))\
    \text{if} \ (t,s) \in L^{n-1} \times I
  \]
  because $\Phi|L^{n-1} = i \circ u$ is $\epsilon$-tame.
  Hence the map $G \colon W \times I \coprod I^n \times I \to Z$,
  which takes $(x,s) \in W \times I$ to $i(x)$ and
  $(t,s) \in I^n \times I$ to $F(t,s)$, factors through the subduction
  $i \times 1 \bigcup \Phi \times 1 \colon W \times I \coprod I^n
  \times I \to Z \to Z \times I$
  (Lemma \ref{lmm:subduction of pushout}), which in turn induces a map
  $K \colon Z \times I \to Z$ making the diagram
  \[
  \xymatrix{%
    W \times I \coprod I^n \times I {%
      \ar[r]^-{G} \ar[d]_-{i \times 1 \bigcup \Phi \times 1} }%
    & Z
    \\
    Z \times I \ar[ru]_-{K} }%
  \]
  commutative.  It is easily verified that the homotopy $K$ satisfies
  $K_0 = 1_Z$, $K_1(Z) \subset W$, and $K_u|W = 1_W$
  ($0 \leq u \leq 1$).  Hence $W$ is a deformation retract of
  $Z = W \cup_{(f,j_n)} I^n$.
\end{proof}

By combining Propositions~\ref{prp:trivial J-fibration is I-fibration}
and \ref{prp:factorization by infinite gluing} we obtain \textbf{MC5}
(i).

\begin{prp}
  \label{prp:MC5(i)}
  In the factorization
  $X \xrightarrow{i_{\infty}} G^{\infty}(\I,f)
  \xrightarrow{p_{\infty}} Y$, the map $i_{\infty}$ is a cofibration
  and $p_{\infty}$ is a trivial fibration.
\end{prp}

On the other hand, Proposition~\ref{prp:factorization by infinite
  gluing} and Lemma~\ref{lmm:deformation retract of cell} imply
\textbf{MC5} (ii).

\begin{prp}
  \label{prp:J-factorization}
  In the factorization
  $X \xrightarrow{i_{\infty}} G^{\infty}(\J,f)
  \xrightarrow{p_{\infty}} Y$, the map $i_{\infty}$ is a trivial
  cofibration and $p_{\infty}$ is a fibration.
\end{prp}

Finally, we verify \textbf{MC4}~(ii) by making use of the preceding
proposition.

\begin{prp}
  Every trivial cofibration has the left lifting property with respect
  to fibrations.
\end{prp}

\begin{proof}
  Suppose $i \colon X \to Y$ is a trivial cofibration and
  $p \colon A \to B$ a fibration.
  Let $f \colon X \to A$ and $g \colon Y \to B$ be smooth maps such
  that $p \circ f = g \circ i$ holds.
  Let us take the factorization
  $i = p_{\infty} \circ i_{\infty} \colon X \to G^{\infty}(\J,i) \to
  Y$, where $i_{\infty}$ is a trivial cofibration and $p_{\infty}$ is
  a fibration.  Because $i$ and $i_{\infty}$ are weak equivalences,
  $p_{\infty}$ is a weak equivalence, and hence a trivial fibration.

  Now, consider the commutative square
  \[
    \xymatrix{%
      X \ar[d]^-{i} \ar[r]^-{i_{\infty}} & G^{\infty}(\J,i)
      \ar[d]^-{p_{\infty}}
      \\
      Y \ar@{=}[r] \ar@{.>}[ru] & Y.
    }%
  \]
  As $i$ is a cofibration, there exists a lift
  $h \colon Y \to G^{\infty}(\J,i)$ such that $p_{\infty} \circ h =1$
  and $ h\circ i = i_{\infty}$ hold (MC4~(i)).
  Hence we obtain a commutative diagram
  \[
    \xymatrix{%
      X \ar[d]^-{i} \ar@{=}[r] & X \ar@{=}[r] \ar[d]^-{i_{\infty}} & X
      \ar[d]^-{i} \ar[r]^-{f} & A \ar[d]^-{p}
      \\
      Y \ar[r]^-{h} & G^{\infty}(\J,i) \ar[r]^-{p_{\infty}}
      \ar@{.>}[rru] & Y \ar[r]^-{g} & B.
    }%
  \]
  As $i_{\infty}$ is a $\J$-cofibration, there exists a lift
  $g^{\prime} \colon G^{\infty}(\J,i) \to A$ making the diagram
  commutative.  Thus we obtain a desired left lift
  $g' \circ h \colon Y \to A$.
\end{proof}

This completes the proof of Theorem \ref{thm:model structure}.

\section{Quillen equivalence between \textbf{Diff} and \textbf{Top}}
We shall show that there exists a Quillen equivalence between the
model categories $\Diff$ and $\Top$.

Let us briefly recall the Quillen model structure on $\Top$.  Let
$D^n$ be the unit $n$-disk in $\R^n$, and let $S^{n-1}$ be the unit
$n$-sphere in $\R^n$.  Let $I'$ be the set of boundary inclusions
$S^{n-1} \to D^n$, $J$ the set of the inclusions
$D^n \times \{0\} \to D^n \times I$, and $W_{\Top}$ the class of weak
homotopy equivalences.

\begin{thm}[{\cite[2.4.19]{Hov}}]
  There exists a finitely generated model structure on $\Top$ with
  $I'$ as the set of generating cofibrations, $J$ as the set of
  generating trivial cofibrations, and $W_{\bf Top}$ as the class of
  weak equivalences.
\end{thm}

It follows that a continuous map $p \colon X \to Y$ between
topological spaces is a fibration in $\Top$ if it has the right
lifting property with respect to the inclusions
$D^n \times \{0\} \to D^n \times I$, and is a trivial fibration if it
has the right lifting property with respect to the inclusions
$S^{n-1} \to D^n$.

There is an adjunction $(T,D,\varphi)$ from $\Diff$ to $\Top$, where

\begin{enumerate}
\item the left adjoint $T \colon \Diff \to \Top$ takes a diffeological
  space to the topological space which has the same underlying set,
  but is equipped with the initial topology with respect to the plots
  of $X$;
\item the right adjoint $D \colon \Top \to \Diff$  takes a
  topological space to the diffeological space which has the same
  underlying set, but is equipped with the
  diffeology consisting of all continuous
  parameterizations;
\item $\varphi$ is a natural isomorphism $\hom_{\Top}(TX,Y) \cong
  \hom_{\Diff}(X,DY)$.
\end{enumerate}

\medskip

The main result of this section is the following.

\begin{thm}
  \label{thm:Quillen equivalence}
  The adjunction $(T,D,\varphi)$ induces a Quillen equivalence between
  the model categories $\Diff$ and $\Top$.
\end{thm}

To see that $(T,D,\varphi)$ is a Quillen adjunction, we need the
following lemma.

\begin{lmm}
  \label{lmm:TI^n is I^n}
  The topological spaces $TI^n$, $T\bI^n$, and $TL^{n-1}$ are
  respectively homeomorphic to the topological subspaces $I^n$,
  $\bI^n$, and $L^{n-1}$ of $\R^n$.
\end{lmm}

\begin{proof}
  It is clear that $T\R^n$ is homeomorphic to the topological space
  $\R^n$.  As $I^n$ is a convex subset of $\R^n$, we conclude that
  $TI^n$ and $I^n$ are homeomorphic by {\cite[Lemma 3.16]{DGE}},
  which in turn implies that $T\bI^n$ and $TL^{n-1}$ are
  homeomorphic to $\bI^n$ and $L^{n-1}$, respectively.
\end{proof}

\begin{prp}
  \label{prp:Quillen adjunction}
  $(T,D,\varphi)$ is a Quillen adjunction from $\Diff$ to $\Top$.
\end{prp}

\begin{proof}
  It suffices to show that the right adjoint $D \colon \Top \to \Diff$
  preserves fibrations and trivial fibrations by {\cite[Lemma
    1.3.4]{Hov}}.
  Suppose $p \colon X \to Y$ is a trivial fibration in $\Top$.  We
  shall show that $Dp \colon DX \to DY$ is a trivial fibration, or
  equivalently, an $\I$-fibration (cf.\ Proposition~\ref{prp:trivial
    J-fibration is I-fibration}).
  Suppose we are given an $\epsilon$-admissible pair given by
  $f \colon \partial I^n \to DX$ and $g \colon I^n \to DY$.
  Then we have a commutative diagram
  \[
    \xymatrix@C=40pt{%
      \partial TI^n \ar[r]^-{Tf} \ar[d]^-{i_n} %
      & TDX \ar[d]^-{TDp} \ar[r]^-{\varepsilon_X} & X \ar[d]^p %
      \\ %
      TI^n \ar[r]^-{Tg} \ar@{.>}[rru] %
      & TDY \ar[r]^-{\varepsilon_Y} & Y. %
    }%
  \]
  Since $(TI^n,T\partial I^n) \cong (D^n,S^{n-1})$ and $p$ is a
  trivial fibration, there exists a lift $G \colon TI^n \to X$
  satisfying $p \circ G = \varepsilon_Y \circ Tg$ and
  $G \circ i_n = \varepsilon_X \circ Tf$.
  Moreover, $G$ can be taken to be $\epsilon^n$-tame by composing, if
  necessary, with the map $c(\epsilon^n)^n \colon I^n \to I^n$, where
  $c(\epsilon^n) \colon I \to I$ is the cut-off function having value
  $\max\{\epsilon^n,\, t\}$ for $t \leq 1/2$ and
  $\min\{1-\epsilon^n,\, t\}$ for $t \geq 1/2$.

  Now, let us define a smooth map $\tilde{g} \colon I^n \to DX$ to be
  the composition
  \[
    I^n \xrightarrow{\eta} DTI^n \cong DI^n \xrightarrow{DG} DX.
  \]
  Then $\tilde{g}$ is an $\epsilon$-admissible lift satisfying
  $\tilde{g} \circ i_n = f$ and $Dp \circ \tilde{g} = g$, showing that
  $Dp \colon DX \to DY$ is an $\I$-fibration.  Hence $D$ preserves
  trivial fibrations.

  By arguing similarly, but with $\bI^n$ replaced $L^{n-1}$, we can
  show that $D$ also preserves fibrations.
\end{proof}


\begin{lmm}
  \label{lmm:cofibration is a retract}
  Any cofibration $i \colon A \to X$ is a retract of
  $i_{\infty} \colon A \to G^{\infty}(\I,i)$, that is, there is a
  commutative diagram
  \begin{equation}
    \label{diagram:retraction induced by cofibration}
    \vcenter{%
      \xymatrix{%
        A \ar@{=}[r] \ar[d]^-{i} & A \ar@{=}[r] \ar[d]^-{i_{\infty}}
        & A \ar[d]^-{i}
        \\
        X \ar[r]^-{h} & G^{\infty}({\I},i) \ar[r]^-{p_{\infty}} & X.
      }%
    }%
  \end{equation}
\end{lmm}

\begin{proof}
  Consider the commutative square
  \[
    \xymatrix{%
      A \ar[d]^-{i} \ar[r]^-{i_{\infty}} & G^{\infty}(\I,i)
      \ar[d]^-{p_{\infty}}
      \\
      X \ar@{=}[r] \ar@{.>}[ru] & X.
    }%
  \]
  Since $i$ is a cofibration and $p_{\infty}$ is a trivial fibration,
  there exists a lift $h \colon X \to G^{\infty}({\I},i)$ making the
  two triangles commutative.  Clearly, this means that the diagram
  (\ref{diagram:retraction induced by cofibration}) is commutative.
\end{proof}

\begin{dfn}
  \label{dfn:gathered I-cell complex}
  A diffeological space $X$ is called an \emph{$\I$-cell complex} if
  the map $\emptyset \to X$ is a relative $\I$-cell complex, and is
  called a \emph{gathered $\I$-cell complex} if all its attaching maps
  are tame.
\end{dfn}

By the definition, $F^l(\I,\emptyset \to X)$ is a gathered $\I$-cell
complex for every $X$ and $1 \leq l \leq \infty$.

\begin{prp}
  \label{prp:reduction to I-cell complexes}
  $(T,D,\varphi)$ is a Quillen equivalence if the natural map
  $X \to DTX$ is a weak equivalence for every gathered $\I$-cell
  complex $X$.
\end{prp}

\begin{proof}
  Since every objects in $\Top$ is fibrant, and since $D$ reflects
  weak equivalences, $(T,D,\varphi)$ is a Quillen equivalence if
  $X \to DTX$ is a weak equivalence for every cofibrant $X$ (cf.\
  \cite[Corollary~1.3.16]{Hov}).
  Thus, to prove the proposition it suffices to show that if
  $X \to DTX$ is a weak equivalence for every gathered $\I$-cell
  complex $X$ then so is for every cofibrant $X$.

  Suppose $X$ is cofibrant, that is, the map
  $i \colon \emptyset \to X$ is a cofibration.  Let
  $Z = G^{\infty}(\I,i)$ and let $h \colon X \to Z$ be the map
  satisfying $p_{\infty} \circ h = 1$ given by
  Lemma~\ref{lmm:cofibration is a retract}.  Then $h$ is a weak
  equivalence because so is $p_{\infty}$.
  Moreover, the map $Z \to DTZ$ is a weak equivalence because
  $G^l(\I,i)$ is homotopy equivalent to the gathered $\I$-cell complex
  $F^l(\I,i)$ and $DT$ preserves homotopies (cf.\
  Proposition~\ref{prp:T and D preserve homotopies}).
  Now, we have a commutative diagram
  \[
    \xymatrix@C=36pt{%
      \pi_n(X,x_0) \ar[r]^-{h_*}_-{\cong} \ar[d] %
      & \pi_n(Z,h(x_0)) \ar[r]^-{p_{\infty *}}_-{\cong}
      \ar[d]^-{\cong} %
      & \pi_n(X,x_0) \ar[d] %
      \\ %
      \pi_n(DTX,x_0) \ar[r]^-{DTh_*} %
      & \pi_n(DTZ,h(x_0)) \ar[r]^-{DTp_{\infty *}} %
      & \pi_n(DTX,x_0) %
    }%
  \]
  in which top horizontal arrows and middle vertical arrow are
  isomorphisms.
  By the commutativity of the left hand square, we see that
  $\pi_n(X,x_0) \to \pi_n(DTX,x_0)$ is a monomorphism.
  On the other hand, as we have
  \[
    DT(p_{\infty})_*\circ DTh_* = DT(p_{\infty} \circ h)_* =
    (1_{DTX})_* = 1_{\pi_n(DTX,x_0)},
  \]
  $DT(p_{\infty})_*$ is an epimorphism, hence so is
  $\pi_n(X,x_0) \to \pi_n(DTX,x_0)$ for any $x_0 \in X$, implying that
  $X \to DTX$ is a weak equivalence.
\end{proof}

Thus, Theorem~\ref{thm:Quillen equivalence} is a consequence of the
next proposition.

\begin{prp}
  \label{prp:the case of I-cell complexes}
  If $X$ is a gathered $\I$-cell complex then $X \to DTX$ is a weak
  equivalence.
\end{prp}

By adjointness, the smooth homotopy groups of $(DTX,x_0)$ are
naturally isomorphic to the continuous homotopy groups of $(TX,x_0)$.
Therefore, Proposition~\ref{prp:the case of I-cell complexes} follows
from the two statements below.
For brevity, we write ``$f \colon Z \to W$ is continuous'' to mean $f$
is a continuous map from $TZ$ to $TW$.  Similarly, if
$f \colon Z \to W$ is continuous and $A$ is a subspace of $Z$ then
``$f$ is tame on $A$'' means that $f|A \colon A \to W$ is smooth and
tame.

\begin{enumerate}
\item Any continuous map $f \colon (I^n,\bI^n) \to (X,x_0)$ is
  continuously homotopic to a tame map
  $g \colon (I^n,\bI^n) \to (X,x_0)$.
\item If there is a continuous homotopy between tame maps
  $f_0$ and
  $f_1$, then there exists a smooth homotopy
  $g \colon (I^n,\bI^n) \times I \to (X,x_0)$ such that $g_0 = f_0$
  and $g_1 = f_1$ hold.
\end{enumerate}
Clearly, these two statements are consequences of the next proposition.

\begin{prp}
  \label{prp:smooth approximation theorem}
  Let $X$ be a gathered $\I$-cell complex, and let
  $f \colon I^n \to X$ be a continuous map.  Then $f$ is continuously
  homotopic to a tame map $g \colon I^n \to X$.  If $f$ is tame on a
  cubical subcomplex $L$ of $I^n$ then the homotopy $f \simeq Tg$ can
  be taken to be relative to $L$.
\end{prp}

To prove this, we require several lemmas.

\begin{lmm}
  \label{lmm:smoothing continuous maps rel L}
  Let $f \colon I^n \to X$ be a continuous map which is tame on a
  cubical subcomplex $L$.  Suppose there is a tame map
  $g' \colon I^n \to X$ and a homotopy $f \simeq Tg'$ restricting to a
  tame homotopy $f|L \simeq g'|L$.
  Then there exists a homotopy $f \simeq Tg$ relative to $L$ such that
  $g \colon I^n \to X$ is tame.
\end{lmm}

\begin{proof}
  It is obvious that there is a homotopy $f \simeq f' \rel{L}$ such
  that $f'$ is tame as a continuous map (cf.\ Lemma~\ref{lmm:taming of
    maps}).  Hence we may assume from the beginning that $f$ is a
  continuous tame map.
  Let $h \colon I^n \times I \to X$ be a continuous homotopy from $f$
  to $g'$ which restricts to a tame homotopy $f|L \simeq g'|L$.  Then,
  by applying Theorem~\ref{thm:WHEP} to the tame map $g'$ and the tame
  homotopy $g'|L \simeq f|L$ given by the inverse to $h|L \times I$,
  we obtain a tame homotopy $g' \simeq g$ such that $g$ is tame and
  satisfies $g|L = f|L$.
  However, the composite homotopy $f \simeq g' \simeq g$ is not a
  homotopy relative to $L$ because its restriction to $L$
  is the composition of $f|L \simeq g'|L$ with its inverse
  $g'|L \simeq f|L$, and hence is not constant.

  To convert $f \simeq g$ into a homotopy relative to $L$, let
  $\mu(t) = \lambda(2t) - \lambda(2t - 1)$ and define
  $h' \colon L \times I \times I \to X$ by
  $h'(v,t,s) = h(v,(1-s)\mu(t))$.  Then $h'$ gives
  $f|L \simeq g'|L \simeq f|L$ on $L \times I \times \{0\}$ and the
  constant homotopy on $L \times I \times \{1\}$.
  Now, let $K = I^n \times \{0,1\} \cup L \times I$ and define
  $H \colon K \times I \to X$ by putting
  $H(v,0,s) = f(v),\ H(v,1,s) = g(v)$, and $H(v,t,s) = h'(v,t,s)$ for
  $v \in L$.
  Then, by the homotopy extension property we can extend $H$ to a
  continuous homotopy $G \colon I^n \times I \times I \to X$.
  Clearly, $G|I^n \times I \times \{1\}$ gives a continuous homotopy
  $f \simeq g \colon I^n \to X$ relative to $L$.
\end{proof}

We also need the classical Whitney approximation on manifolds.

\begin{lmm}[{\cite[Theorem 10.21]{Lee}}]
  \label{lmm:Whitney approximation on Manifolds}
  Let $N$ and $M$ be a smooth manifolds, and let $F \colon N \to M$ be
  a continuous map.  Then $F$ is homotopic to a smooth map
  $\widetilde{F} \colon N \to M$.  If $F$ is smooth on a closed subset
  $L \subset N$, then the homotopy can be taken to be relative to $L$.
\end{lmm}

We are now ready to prove Proposition~\ref{prp:smooth approximation
  theorem}.

\begin{proof}[Proof of Proposition~\ref{prp:smooth approximation
    theorem}] 
  Let $f \colon I^n \to X$ be a continuous map which is tame on a
  subcomplex $K \subset I^n$.  As $I^n$ is compact, $f(I^n)$ is
  contained in a finite subcomplex.  Thus we may assume $X$ is a
  finite $\I$-cell complex.

  Let $m$ be the number of cells of $X$, so that there is a finite
  sequence
  \[
    \emptyset = X_0 \to X_1 \to \cdots \to X_m = X
  \]
  such that for each $k \leqq m$ there are an integer $d_k \geq 0$ and
  a pushout square
  \[
    \vcenter{%
      \xymatrix{%
        \bI^{d_k} \ar[d] \ar[r]^-{\phi_k} & X_{k-1} \ar[d]
        \\
        I^{d_k} \ar[r]^-{\Phi_k} & X_k.  }%
    }%
  \]
  We shall prove the statement by induction on $m$.  If $m = 1$ then
  the statement holds because $X_1$ is a point.
  Let $m > 1$ and suppose the statement is true for any gathered
  $\I$-cell complex having less than $m$ cells.
  If $d_m = 0$ then $X$ is the disjoint union $X_{m-1} \coprod I^0$
  and the statement surely holds.
  Suppose $d_m > 0$ and $\phi_m$ is $\epsilon$-tame.  If we put
  $r = (T_{\epsilon/2,\epsilon})^{d_m} \colon I^{d_m} \to I^{d_m}$
  then we have $\phi_m((1-t)v + t\,r(v)) = \phi(v)$ for
  $(v,t) \in \bI^{d_m} \times I$ because $\phi_m$ is $\epsilon$-tame.
  Hence we can define $h \colon X \times I \to X$ by the formula
  \[
  h(x,t) =
  \begin{cases}
    \Phi_m((1-t)v + t\,r(v)), & x = \Phi_m(v) \in \Phi_m(I^{d_m})
    \\
    x, & x \in X_{m-1}.
  \end{cases}
  \]
  As $r$ is $\epsilon/2$-tame, we see that $h$ is smooth.
  Now, let
  \[
    U = \Phi_m(I^{d_m} - \bI^{d_m}) \ \ \text{and}\ \
    V = \Phi_m(I^{d_m} - [\epsilon,1 - \epsilon]^{d_m}) \cup X_{m-1}.
  \]
  Then we have the following.
  \begin{enumerate}
  \item The set $\{U,\, V\}$ is an open cover of $X$.
  \item $U$ is diffeomorphic to the euclidean space $\R^{d_m}$.
  \item There is a smooth homotopy $h \colon X \times I \to X$
    relative to $X_{m-1}$ such that $h_0$ is the identity and $h_1$
    gives a retraction of $V$ onto $X_{m-1}$.
  \end{enumerate}
  Let $\Sd_k(I^n)$ be the cubical subdivision of $I^n$ consisting of
  subcubes
  \[
  K_J = \left[\frac{j_1 - 1}{k},\frac{j_1}{k}\right] \times \cdots
  \times \left[\frac{j_n - 1}{k},\frac{j_n}{k}\right]
  \]
  where $J = (j_1,\, \cdots,\, j_n) \in \{1,\cdots,k\}^n$.
  By taking $k$ large enough, we may assume each $f(K_J)$ is contained
  in either $U$ or $V$.

  We construct a homotopy from $f$ to a tame map $g$ in several steps.

  \smallskip

  \emph{Step 1.} We construct a homotopy $f \simeq f'$ which restricts
  to a tame homotopy $f|L \simeq f'|L$ and is such that $f'$ is tame
  not only on $L$ but also on every subface $L \cap K_J$.
  Let $f'$ be the composition $f' = f \circ \xi$, where
  $\xi \colon I^n \to I^n$ is a tame map such that for every $J$ the
  restriction $\xi|K_J$ is identical with $T_{\sigma,\tau}^n$ under
  the evident homeomorphism $K_J \cong I^n$.  Then $f'$ is tame on
  $L \cap K_J$ and there is a homotopy $f \simeq f'$ induced by the
  smooth homotopy $1 \simeq \xi$ arising from a suitably chosen tame
  homotopy $1 \simeq T_{\sigma,\tau}$.

  \smallskip

  \emph{Step 2.} Let $M$ be the set of labels $J$ such that $f'(K_J)$
  is contained in $U$.  For every subset $M'$ of $M$, write
  $K(M') = \bigcup_{J \in M'} K_J \subset \Sd_k(I^n)$.
  We shall inductively construct a continuous homotopy
  \begin{equation}
    \label{eq:2-1}
    f'|K(M) \simeq g'_M \colon K(M) \to U \ \rel{L \cap K(M)}
  \end{equation}
  such that $g'_M$ is piecewise tame, that is, restricts to a tame map
  $K_J \to U$ for every $J \in M$.  Evidently, this means that $g'_M$
  is smooth all over $K(M)$.

  Let us endow $M$ with the lexicographical order, and denote
  \[
  M^J_0 = \{J' \in M \mid J' < J\}, \quad M^J = M^J_0 \cup \{J\}.
  \]
  for every $J \in M$.  Suppose there is a continuous homotopy
  \begin{equation}
    \label{eq:2-2}
    f'|K(M^J_0) \simeq g'_{M^J_0} \colon K(M^J_0) \to U \ \rel{L \cap K(M^J_0)}
  \end{equation}
  such that $g'_{M^J_0}$ is piecewise tame.
  To see that (\ref{eq:2-1}) exists, it suffices to show that the
  homotopy (\ref{eq:2-2}) can be extended to a homotopy
  \begin{equation}
    \label{eq:2-3}
    f'|K(M^J) \simeq g'_{M^J} \colon K(M^J) \to U \ \rel{L \cap K(M^J)}
  \end{equation}
  such that $g'_{M^J}$ is piecewise tame.
  By the homotopy extension property with respect to
  $(K_J, L \cap K_J \cap K(M^J_0))$, there exists a continuous homotopy
  \begin{equation*}
    f'|K_J \simeq \tilde{f}'_J \colon K_J \to U \ \rel{L \cap
      K_J},
  \end{equation*}
  which coincides with (\ref{eq:2-2}) on the subface
  $K(M^J_0) \cap K_J$.
  But as $U$ is diffeomorphic to $\R^{d_m}$, we can apply the Whitney
  approximation theorem (Lemma~\ref{lmm:Whitney approximation on
    Manifolds}) to get a homotopy
  \begin{equation}
    \label{eq:2-4}
    \tilde{f}'_J \simeq g'_J \colon K_J \to U \ \rel{(L \cup
      K(M^J_0)) \cap K_J}
  \end{equation}
  such that $g'_J$ is smooth.  Moreover, $g'_J$ can be taken to be
  tame by composing, if necessary, with (suitably rescaled)
  $T_{\sigma,\tau}^{d_m} \colon K_J \to K_J$ with $\tau$ sufficiently
  small.
  Thus, we get (\ref{eq:2-3}) by pasting (\ref{eq:2-2}) and
  (\ref{eq:2-4}) together, and hence (\ref{eq:2-1}) exists.

  \smallskip

  \emph{Step 3.} Let us write $\rho = h_1 \colon X \to X$, and put
  $f'' = \rho\circ f'$, $g''_M = \rho\circ g'_M$.  By Step 2, there is a
  continuous homotopy
  \begin{equation}
    \label{eq:3-1}
    f''|K(M) \simeq g''_M \colon K(M) \to X \ \rel{L \cap K(M)}
  \end{equation}
  obtained from (\ref{eq:2-1}) 
  by composing with $\rho$.

  If $J$ is not contained in $M$ then $f'(K_J)$ is contained in $V$,
  and hence $f''(K_J) = \rho(f'(K_J))$ is contained in $X_{m-1}$.  By
  the inductive assumption, we can apply Proposition~\ref{prp:smooth
    approximation theorem} to $f''|K_J$ (with $I^n$ and $L$ replaced
  by $K_J$ and $L \cap K_J$, respectively).
  Thus, by arguing as in the previous step, but using the assumption
  of the induction instead of the Whitney approximation, we can extend
  (\ref{eq:3-1}) to a continuous homotopy
  \begin{equation}
    \label{eq:3-2}
    f'' \simeq g'' \colon Sd_k(I^n) \to X \ \rel{L}
  \end{equation}
  such that $g''$ is piecewise tame.

  \smallskip

  \emph{Step 4.}  Thus we have a composite homotopy
  $f \simeq f' \simeq f'' \simeq g''$ such that $g''$ is piecewise
  tame with respect to $Sd_k(I^n)$.  But every piecewise tame map from
  $Sd_k(I^n)$ is tame as a map from $I^n$.  Hence we may regard $g''$
  as a tame map $I^n \to X$.  Moreover, the homotopy $f \simeq g''$ we
  have constructed restricts to a tame homotopy
  $f|L \simeq \rho \circ f|L$ on $L$.  Hence there exist by
  Lemma~\ref{lmm:smoothing continuous maps rel L} a tame map $g$ and a
  continuous homotopy $f \simeq g \rel{L}$.

  This completes the proof of the proposition.
\end{proof}

\section{Comparison with the model category $\NumG$}

Let $\NumG$ be the full subcategory of $\Top$ consisting of those
objects $Y$ such that the counit $TDY \to Y$ is a homeomorphism, and
let $\STop$ be the full subcategory of $\Diff$ consisting of those $X$
such that the unit $X \to DTX$ is a diffeomorphism.
As we have shown in \cite{SYH}, $\NumG$ can be identified with the
full subcategory of $\Top$ consisting of numerically generated (same
as $\Delta$-generated) spaces;
while on the other hand, $\STop$ can be identified with the category
of topological spaces with numerically continuous maps as morphisms.
As in \cite[Section~4]{SYH}, we denote by $\smap(X,Y)$ the set of
numerically continuous maps $X \to Y$ equipped with the topology such
that the following holds (Proposition~4.7 of \cite{SYH}):
\[
D\smap(X,Y) = \Cinfty(DX,DY).
\]
The category $\NumG$ is cartesian closed with respect to exponentials
$Y^X = \nu\smap(X,Y)$, where $\nu$ denotes the coreflection $TD \colon
\Top \to \NumG$.

It is clear that $T$ factors as a composition
$\Diff \to \NumG \to \Top$, and $D$ factors as
$\Top \to \STop \to \Diff$.  Clearly, $T$ and $D$ induce an inverse
equivalence between $\STop$ and $\NumG$.

\begin{prp}
  \label{prp:T and D preserve homotopies}
  Both $T \colon \Diff \to \NumG$ and $D \colon \Top \to \STop$
  preserve homotopies.
\end{prp}

\begin{proof}
  If $f \colon X \times I \to Y$ is a smooth homotopy between $f_0$
  and $f_1$ then we have a composition
  \[
  TI \xrightarrow{TF} T\Cinfty(X,Y) \to T\Cinfty(DTX,DTY) =
  \nu\smap(TX,TY) 
  \]
  where $F \colon I \to \Cinfty(X,Y)$ is the adjoint to $f$.  Since
  $\nu\smap(TX,TY) = TY^{TX}$ is the exponential object in $\NumG$, we
  obtain, by adjunction, a continuous homotopy
  $TX \times TI \cong TX \times [0,1] \to TY$ between $Tf_0$ and
  $Tu_1$.

  On the other hand, if $g \colon Z \times TI \to W$ is a continuous
  homotopy between $g_0$ and $g_1$ then the composition
  \[
  DZ \times I \to DZ \times DTI = D(Z \times TI) \xrightarrow{Dg} DW
  \]
  gives a smooth homotopy between $Dg_0$ and $Dg_1$.
\end{proof}

\begin{rmk}
  The inclusion $\STop \to \Diff$ also preserves homotopies, hence so
  does $D \colon \Top \to \Diff$.  But the situation is subtle for
  $\NumG \to \Top$, as it does not commute with product nor $\Top$ is
  not cartesian closed.
\end{rmk}

By \cite[Theorem~3.3]{Haraguchi}, $\NumG$ has a finitely generated
model structure that is Quillen equivalent to that of $\Top$ under the
adjunction $(i,\nu)$, where $i$ is the inclusion of $\NumG$ into
$\Top$ and $\nu = TD$ is the coreflection $\Top \to \NumG$.
Thus, there is a sequence of Quillen equivalences between model
categories
\[
\xymatrix{%
  \Diff \ar@<0.5ex>[r]^-{T} &
  \NumG \ar@<0.5ex>[l]^-{D} \ar@<0.5ex>[r]^-{i} &
  \Top. \ar@<0.5ex>[l]^-{\nu}
}%
\]

As a final note of the paper, we remark that there are non-cofibrant,
but still geometrically interesting, diffeological spaces $X$ which
have smooth homotopy groups not isomorphic to their continuous
homotopy groups.
Note that for such an $X$ the map $X \to DTX$ is not a weak
equivalence, and hence it does not have the smooth homotopy type of a
topological space.

To illustrate the situation, take the irrational torus
$\mathbb{T}_{\theta}$.  We have shown in Example~\ref{ex:irrational
  torus} that the smooth fundamental group of $\mathbb{T}_{\theta}$ is
isomorphic to $\mathbb{Z}^2$; while on the other hand,
$T\mathbb{T}_{\theta}$ has the trivial fundamental group because it is
an indiscrete topological space.  From this we observe the following.

\begin{prp}
  \label{prp:type of irrational torus}
  The natural map $\mathbb{T}_{\theta} \to DT\mathbb{T}_{\theta}$ is
  not a weak equivalence.  Moreover, $\mathbb{T}_{\theta}$ does not
  have the smooth homotopy type of a topological space.
\end{prp}

\begin{proof}
  For a topological space $Y$ the smooth fundamental group of $DY$ is
  isomorphic to the continuous fundamental group of $Y$.  In
  particular, by taking $Y = T\mathbb{T}_{\theta}$, we see that the
  smooth fundamental group of $DT\mathbb{T}_{\theta}$ is trivial.
  Therefore, $\mathbb{T}_{\theta} \to DT\mathbb{T}_{\theta}$ cannot be
  a weak homotopy equivalence.

  To prove the second assertion, suppose there is a homotopy
  equivalence $DY \simeq \mathbb{T}_{\theta}$ for some $Y \in \Top$.
  Since $D$ preserves homotopy groups, and since $TDY \to Y$ is a weak
  homotopy equivalence by \cite[Proposition~5.4]{SYH}, we have
  $\pi_1(TDY) \cong \mathbb{Z}^2$.
  But, on the other hand, we have $TDY \simeq T\mathbb{T}_{\theta}$
  because $T$ preserves homotopy equivalences (Proposition~\ref{prp:T
    and D preserve homotopies}).  Hence $TDY$ have the trivial
  fundamental group, in contradiction to the previous assertion that
  $TDY$ must have the non-trivial fundamental group.
\end{proof}

\providecommand{\bysame}{\leavevmode\hbox to3em{\hrulefill}\thinspace}
\providecommand{\MR}{\relax\ifhmode\unskip\space\fi MR }
\providecommand{\MRhref}[2]{%
  \href{http://www.ams.org/mathscinet-getitem?mr=#1}{#2}
}
\providecommand{\href}[2]{#2}

\end{document}